\documentclass[11pt]{elsarticle}

\usepackage[ddmmyyyy]{datetime}
\usepackage{geometry}
\usepackage{lineno,hyperref}
\nolinenumbers
\usepackage{float}
\usepackage{amsfonts}
\usepackage{mathrsfs}
\usepackage{amssymb}
\usepackage{amsmath}
\usepackage{amsthm}
\usepackage{appendix}
\usepackage{xcolor}
\usepackage{graphicx}
\usepackage[font=small,labelfont=bf]{caption}
\usepackage{subcaption}
\usepackage{algpseudocode}
\usepackage{srfdisN}
\usepackage{srfdisNmicros}

\begin{document}

\begin{frontmatter}

\title{Needlet approximation for isotropic random fields on the sphere\tnoteref{tifn}}
\tnotetext[tifn]{This research was supported under the Australian Research Council's \emph{Discovery Projects} DP120101816 and DP160101366.
The first author was supported under the University International Postgraduate Award (UIPA) of UNSW Australia.}

\author[addunsw]{Quoc T. Le Gia\corref{corau}}
\ead{qlegia@unsw.edu.au}
\author[addunsw]{Ian H. Sloan}
\ead{i.sloan@unsw.edu.au}
\author[addunsw,addcityu]{Yu Guang Wang\fnref{fn1}\corref{corau}}
\ead{y.wang@latrobe.edu.au}
\author[addunsw]{Robert S. Womersley}
\ead{r.womersley@unsw.edu.au}

\cortext[corau]{Corresponding author.}

\fntext[fn1]{Present address: Department of Mathematics and Statistics, La Trobe University, Bundoora 3086 VIC, Australia.}

\address[addunsw]{School of Mathematics and Statistics, UNSW Australia, Sydney, NSW, 2052, Australia}
\address[addcityu]{Department of Mathematics, City University of Hong Kong, Tat Chee Avenue, Kowloon Tong, Hong Kong}

\begin{abstract}
In this paper we establish a multiscale approximation for random fields on the sphere using spherical needlets --- a class of spherical wavelets. We prove that the semidiscrete needlet decomposition converges in mean and pointwise senses for weakly isotropic random fields on $\sph{d}$, $d\ge2$. For numerical implementation, we construct a fully discrete needlet approximation of a smooth $2$-weakly isotropic random field on $\sph{d}$ and prove that the approximation error for fully discrete needlets has the same convergence order as that for semidiscrete needlets. Numerical examples are carried out for fully discrete needlet approximations of Gaussian random fields and compared to a discrete version of the truncated Fourier expansion.
\end{abstract}


\end{frontmatter}

\section{Introduction}\label{sec:intro}
Isotropic random fields on the sphere have application in environmental models and astrophysics, see \cite{CaSt2013,Durrer2008,LaGu1999,RuReMe2013,Stein2007,StChAn2013}.
It is well-known that a $2$-weakly isotropic random field on the unit sphere $\sph{d}$ in $\Rd[d+1]$, $d\ge2$, has a Karhunen-Lo\`{e}ve (K-L) expansion in terms of spherical harmonics, see e.g. \cite{LaSc2015,MaPe2011}.
In this paper we establish a multiscale approximation for spherical random fields using spherical needlets \cite{NaPeWa2006-1,NaPeWa2006-2}. We prove that the semidiscrete needlet approximation converges in both mean and pointwise senses for a weakly isotropic random field on $\sph{d}$. We establish a fully discrete needlet approximation by discretising the needlet coefficient integrals by means of suitable quadrature (numerical integration) rules on $\sph{d}$. We prove that when the field is sufficiently smooth and the quadrature rule has sufficiently high polynomial degree of precision the error of the approximation by fully discrete needlets has the same convergence order as that for semidiscrete needlets. An algorithm for fully discrete needlet approximations is given, and numerical examples using Gaussian random fields are provided.

Let $(\probSp,\mathcal{F},\probm)$ be a probability measure space. For $1\le p\le\infty$, let $\Lpprob[{,\probm}]{p}$ be the $\mathbb{L}_{p}$-space on $\probSp$ with respect to the probability measure $\probm$, endowed with the norm $\norm{\cdot}{\Lpprob{p}}$. Let $\expect{\ranva}$ denote the \emph{expected value} of a random variable $\ranva$ on $(\probSp,\mathcal{F},\probm)$.

\textbf{Random fields.} A \emph{real-valued random field} on the sphere $\sph{d}$ is a function $T:\probSp\times\sph{d}\to\mathbb{R}$. We write $\RF(\omega,\PT{x})$ by $\RF(\PT{x})$ or $\RF(\omega)$ for brevity if no confusion arises.
We say $\RF$ is $2$-weakly isotropic if its expected value and covariance are rotationally invariant.

%\textbf{Needlets.}
%Let $\mathbb{R}_{+}:=[0,\infty)$. A \emph{needlet filter} is a continuous compactly supported function $\fiN: \mathbb{R}_{+}\to\mathbb{R}_{+}$ with support a subset of $[1/2,2]$
%and satisfying
%  \begin{equation*}
%  \fiN\in \CkR, \quad \fiN(t)^{2} + \fiN(2t)^{2} = 1 \;\hbox{~for~} t\in [1/2,1].
%  \end{equation*}
\emph{Spherical needlets} \cite{NaPeWa2006-1,NaPeWa2006-2} $\needlet{jk}$ are localised polynomials on the sphere associated with a \emph{quadrature rule} and a \emph{filter}; see \eqref{subeqs:fiN}--\eqref{subeqs:needlets} below. The level-$j$ needlet $\needlet{jk}$, $k=1,\dots,N_{j}$, is a spherical polynomial of degree $2^{j}-1$, associated with a positive weight quadrature rule with degree of precision $2^{j+1}-1$.\\[2mm]
\textbf{Main results.}\\[-2mm]

Needlets form a multiscale tight frame for square-integrable functions on $\sph{d}$, see \cite{NaPeWa2006-1,NaPeWa2006-2}. The resulting tight frame has good approximation performance for random fields on the sphere.

\textbf{Needlet decomposition for random fields.} For a random field $\RF$ on $\sph{d}$, the \emph{(semidiscrete) needlet approximation} of order $\neord$ for $\neord\in\Nz$ is defined as
\begin{equation}\label{eq:intro RF via needlet}
  \neapx(\RF;\omega,\PT{x}) :=  \sum_{j=0}^{\neord}\sum_{k=1}^{N_{j}}\InnerL{\RF(\omega),\needlet{jk}}\needlet{jk}(\PT{x}),\quad \omega\in\probSp,\;\PT{x}\in\sph{d}.
\end{equation}

Let $\ceil{\cdot}$ be the ceiling function. Given $1\le p<\infty$ and $d\ge2$, the needlet approximation given by \eqref{eq:intro RF via needlet} of a $\ceil{p}$-weakly isotropic random field $\RF$ on $\sph{d}$ converges to $\RF$ in $\Lppsph{p}{d}$.
(See Theorem~\ref{thm:neapx.Lp.converge.RF}, and see Section~\ref{sec:pre} for the definition of $\ordiRF$-weakly isotropic, $\ordiRF\in\N$.)

In Theorem~\ref{thm:Bd.neapx.Lppsph}, we also prove that when $p$ is a positive integer the needlet approximation on the set of $p$-weakly isotropic random fields is bounded in the $\Lppsph{p}{d}$ sense.

Let $d\ge2$. Let $\RF$ be a $2$-weakly isotropic random field satisfying $\sum_{\ell=1}^{\infty}\; \APS{\ell} \ell^{2s+d-1} < \infty$.
Here $\APS{\ell}$ is the angular power spectrum of $\RF$, see Section~\ref{subsec:angular.power.spectrum}.
Then $\RF(\omega)$ is in the Sobolev space $\sob{2}{s}{d}$ $\probm$-almost surely (or $\Pas$). (This fact was proved in \cite[Section~4]{LaSc2015} and \cite{AnLa2014}; we restate this result in Corollary~\ref{cor:APS.RF.in.sob}.)

\textbf{Semidiscrete needlet approximation.}  In Theorems~\ref{thm:err.mean.L2.neapx} and \ref{thm:err.pointwise.L2.neapx}, we prove that the semidiscrete needlet approximation $\neapx(\RF)$ has the following mean and pointwise approximation errors. Let $s>0$.
Then
\begin{equation}\label{eq:intro.err.mean.L2.neapx}
 \expect{\normb{\RF-\neapx(\RF)}{\Lp{2}{d}}^{2}}
 \le c\: 2^{-2\neord s}\: \expect{\norm{\RF}{\sob{2}{s}{d}}^{2}}.
\end{equation}
The needlet approximation $\neapx(\RF;\omega)$ has the following $\Lp{2}{d}$-error bound: $\Pas$,
\begin{equation}\label{eq:intro.err.pw.L2.neapx}
  \normb{\RF-\neapx(\RF)}{\Lp{2}{d}} \le  c\:{2^{-\neord s}}\:\norm{\RF}{\sob{2}{s}{d}}  < \infty,
\end{equation}
where the constants $c$ in \eqref{eq:intro.err.mean.L2.neapx} and \eqref{eq:intro.err.pw.L2.neapx} depend only on $d$, $s$, $\fiN$ and $\fis$.

\textbf{Fully discrete needlet approximation.} For implementation, we discretise the needlet coefficient $\InnerL{\RF(\omega),\needlet{jk}}=\int_{\sph{d}}\RF(\omega,\PT{x})\needlet{jk}(\PT{x})\IntDiff{x}$ by a positive weight numerical integration rule,
\begin{equation*}%\label{eq:intro discrete inner prod}
  \InnerD{\RF,\needlet{jk}}:=\sum_{i=1}^{N}\wH \:\RF(\pH{i}) \needlet{jk}(\pH{i}).
\end{equation*}
This then gives the \emph{fully discrete needlet approximation} for the random field $\RF$:
\begin{equation*}
    \disneapx(\RF;\omega,\PT{x}) := \sum_{j=0}^{\neord}\sum_{k=1}^{N_{j}}\InnerD{\RF(\omega),\needlet{jk}} \needlet{jk}(\PT{x}),\quad \omega\in\probSp,\;\PT{x} \in \sph{d}.
\end{equation*}

In Theorems~\ref{thm:err.mean.L2.disneapx} and \ref{thm:err.pointwise.L2.disneapx}, we prove that for $s>d/2$, the fully discrete needlet approximation $\disneapx(\RF)$ has the same mean and pointwise error orders as those for the semidiscrete needlet approximation $\neapx(\RF)$,
cf. \eqref{eq:intro.err.mean.L2.neapx}--\eqref{eq:intro.err.pw.L2.neapx}.
Let $\QH$ be a discretisation quadrature exact for polynomials of degree $3\cdot 2^{\neord-1}-1$. Then
\begin{equation}\label{eq:intro err discrete needlet mean}
 \expect{\normb{\RF-\disneapx(\RF)}{\Lp{2}{d}}^{2}}
 \le c\: 2^{-2\neord s}\: \expect{\norm{\RF}{\sob{2}{s}{d}}^{2}},
\end{equation}
and $\Pas$
\begin{equation}\label{eq:intr.disne.ptw.err}
  \normb{\RF-\disneapx(\RF)}{\Lp{2}{d}}
 \le c\:{2^{-\neord s}}\:\norm{\RF}{\sob{2}{s}{d}} < \infty,
\end{equation}
where the constants $c$ in \eqref{eq:intro err discrete needlet mean} and \eqref{eq:intr.disne.ptw.err} depend only on $d$, $s$, $\fiN$ and $\kappa$.

These results show that the discretisation of the needlet decomposition does not affect the convergence order of the approximation error for a sufficiently smooth random field. This is also supported by numerical examples in Section~\ref{sec:numer.examples}.

In Theorem~\ref{thm:Bd.disneapx.Lppsph}, we prove that when $p$ is a positive integer the fully discrete needlet approximation is bounded on the set of $p$-weakly isotropic random fields in the $\Lppsph{p}{d}$ sense, in a similar way to semidiscrete needlets.

The paper is organised as follows. Section~\ref{sec:pre} makes necessary preparations. Section~\ref{sec:needlet.decomp.RF} shows the convergence of the needlet decomposition for $\ceil{p}$-weakly, $1\le p\le \infty$, isotropic random fields on $\sph{d}$, and proves for $p\in \N$ the boundedness of the semidiscrete needlet approximation on the set of $p$-weakly isotropic random fields in $\Lppsph{p}{d}$. In Section~\ref{sec:truncate err}, we establish a fully discrete needlet approximation for isotropic random fields on the sphere and then prove convergence rates for mean and pointwise approximation errors of semidiscrete and fully discrete needlet approximations for smooth isotropic random fields on $\sph{d}$. In Section~\ref{subsec:Bd.disneapx.Lppsph}, we prove for $p\in\N$ the boundedness of the fully discrete needlet approximation on the set of $p$-weakly isotropic random fields in $\Lppsph{p}{d}$. In Section~\ref{sec:converge.cov.neapx.RF}, we prove convergence rates of variances for both semidiscrete and fully discrete needlet approximations of $2$-weakly isotropic random fields. In Section~\ref{sec:numer.examples}, we give an algorithm for discrete needlet approximations and present numerical examples for Gaussian random fields on $\sph{2}$. We also compare numerically the discrete needlet approximation with a discrete version of the truncated Fourier series approximation of Gaussian random fields.

\section{Preliminaries}\label{sec:pre}
For $d\geq2$, let $\mathbb{R}^{d+1}$ be the real
($d+1$)-dimensional Euclidean space with inner product $\PT{x}\cdot\PT{y}$
for $\PT{x},\PT{y}\in \REuc$ and Euclidean norm
$|\PT{x}|:=\sqrt{\PT{x}\cdot\PT{x}}$. Let
    $\sph{d}:=\{\PT{x}\in\REuc: |\PT{x}|=1\}$
denote the unit sphere of $\REuc$. The sphere $\sph{d}$ forms a compact
metric space, with the metric being the geodesic distance
  $\dist{\PT{x},\PT{y}}:=\arccos(\PT{x}\cdot\PT{y})$ for $\PT{x},\PT{y}\in \sph{d}$.
For $1\le p<\infty$, let $\Lp{p}{d}=\Lp[,\sigma_{d}]{p}{d}$ be the real-valued $\mathbb{L}_{p}$-function space with respect to the normalised Riemann surface measure $\sigma_{d}$ on $\sph{d}$, endowed with the $\mathbb{L}_{p}$ norm
\begin{equation*}
  \norm{f}{\Lp{p}{d}}:=\left\{\int_{\sph{d}}|f(\PT{x})|^{p}\IntDiff{x}\right\}^{1/p},\quad f\in \Lp{p}{d},\;\; 1\le p<\infty.
\end{equation*}
For $p=\infty$, let $\Lp{\infty}{d}:=\ContiSph{}$ be the space of real-valued continuous functions on $\sph{d}$ with uniform norm
\begin{equation*}
\norm{f}{\Lp{\infty}{d}}:= \sup_{\PT{x}\in \sph{d}}|f(\PT{x})|,\quad f\in \ContiSph{}.
\end{equation*}
For $p=2$, $\Lp{2}{d}$ forms a Hilbert space with inner product
  $\InnerL{f,g}:=\int_{\sph{d}}f(\PT{x})g(\PT{x})\IntDiff{x}$ for $f,g\in\Lp{2}{d}$.
Let $\Lppsph{p}{d}:=\Lppsph[,\prodpsphm]{p}{d}$ be the real $\mathbb{L}_{p}$-space on the product space of $\probSp$ and $\sph{d}$ where $\prodpsphm$ is the corresponding product measure.

Let $\varphi$ be a real convex function on $\mathbb{R}$ and let $\ranva$ be a random variable on $\probSp$. We will use Jensen's inequality
\begin{equation}\label{eq:Jensen.ineq.probab}
    \varphi\left(\expect{\ranva}\right) \le \expect{\varphi(\ranva)},
\end{equation}
see e.g. \cite[Eq.~21.14]{Billingsley1995}.

\subsection{Isotropic random fields}
Let $\mathscr{B}(\sph{d})$ denote the Borel algebra on $\sph{d}$ and let $\RotGr$ be the rotation group on $\REuc$.
An $\mathcal{F}\otimes\mathscr{B}(\sph{d})$-measurable function $T:\probSp\times\sph{d}\to\mathbb{R}$ is said to be a \emph{real-valued random field} on the sphere $\sph{d}$. We write $\RF(\omega,\PT{x})$ as $\RF(\PT{x})$ or $\RF(\omega)$ for brevity if no confusion arises.
We say $\RF$ is \emph{strongly isotropic} if for any $k\in\N$ and for all sets of $k$ points $\PT{x}_{1},\cdots,\PT{x}_{k}\in\sph{d}$ and for any rotation $\rho\in \RotGr$, $\RF(\PT{x}_{1}), \dots,\RF(\PT{x}_{k})$ and $\RF(\rho\PT{x}_{1}),\dots,\RF(\rho\PT{x}_{k})$ have the same law. For an integer $\ordiRF\ge1$, we say $\RF$ is $\ordiRF$-\emph{weakly isotropic} if for all $\PT{x}\in \sph{d}$ the $\ordiRF$th moment of $\RF(\PT{x})$ is finite, i.e.
  $\expect{|\RF(\PT{x})|^{\ordiRF}}<\infty$
and if for $k=1,\dots,\ordiRF$, for all sets of $k$ points $\PT{x}_{1},\dots,\PT{x}_{k}$ $\in \sph{d}$ and for any rotation $\rho\in\RotGr$,
\begin{equation*}
  \expect{\RF(\PT{x}_{1})\cdots \RF(\PT{x}_{k})}=\expect{\RF(\rho\PT{x}_{1})\cdots \RF(\rho\PT{x}_{k})},
\end{equation*}
see e.g. \cite{Adler2009,LaSc2015,MaPe2011}.
We define $\RF$ to be an $\infty$-\emph{weakly isotropic} random field if $\RF$ is $\ordiRF$-weakly isotropic for each positive integer $\ordiRF$ and if $\expect{\norm{\RF}{\Lpprob{\infty}}}<\infty$ and if $\RF(\omega)$ is $\Pas$ continuous on $\sph{d}$.

We can also define isotropy for a set of random fields. We say a collection of random fields $\{T_{j}:j\in \Nz\}$ is an \emph{$\ordiRF$-weakly isotropic set} of random fields if $\expect{|T_{j}(\PT{x})|^{\ordiRF}}<\infty$ for every $j\in\Nz$ and for each $\PT{x}\in \sph{d}$, and if for $1\le k\le \ordiRF$, any $k$ different integers $j_{1},\dots,j_{k}$ at most $\ordiRF$, any $k$ points $\PT{x}_{1},\dots,\PT{x}_{k}\in\sph{d}$ and any rotation $\rho\in\RotGr$,
\begin{equation*}
  \expect{T_{j_{1}}(\PT{x}_{1})\cdots T_{j_{k}}(\PT{x}_{k})}=\expect{T_{j_{1}}(\rho\PT{x}_{1})\cdots T_{j_{k}}(\rho\PT{x}_{k})}.
\end{equation*}

We say $\RF$ is a \emph{Gaussian random field} on $\sph{d}$ if the vector $(\RF(\PT{x}_1),\dots,\RF(\PT{x}_{k}))$ has, for each $k\in\N$ and $\PT{x}_{1}, \dots,\PT{x}_{k}\in \sph{d}$, a multivariate Gaussian distribution, see e.g. \cite{Adler2009,LaSc2015,MaPe2011}.

 We have the following properties of isotropic random fields, see e.g. \cite[Proposition~5.10(2)(3), p.~122]{MaPe2011}.
\begin{proposition}\label{prop:isotrRF.weak.strong.Gauss}
(i) Let $\RF$ be an $\infty$-weakly isotropic random field on $\sph{d}$. Let the distribution of $\RF(\PT{x}_{1}),\dots,\RF(\PT{x}_{k})$, for each $k\in \N$ and $\PT{x}_{1},\dots,\PT{x}_{k}$, be determined by the moments of $\RF(\PT{x}_{1}),\dots,\RF(\PT{x}_{k})$. Then $\RF$ is strongly isotropic.\\
(ii) Let $\RF$ be a Gaussian random field on $\sph{d}$. Then $\RF$ is strongly isotropic if and only if $\RF$ is $2$-weakly isotropic.
\end{proposition}

In this paper, for $p=\infty$ (only), we assume the condition of (i) of Proposition~\ref{prop:isotrRF.weak.strong.Gauss} (that the distribution of $\RF(\PT{x}_{1}),\dots,\RF(\PT{x}_{k})$, for each $k\in \N$ and $\PT{x}_{1},\dots,\PT{x}_{k}$, is determined by the moments of $\RF(\PT{x}_{1}),\dots,\RF(\PT{x}_{k})$), and hence that $\RF$ is strongly isotropic.

\subsection{Spherical harmonics}
A \emph{spherical harmonic} of degree $\ell$ on $\sph{d}$ is the restriction to $\sph{d}$ of a homogeneous and harmonic polynomial of total degree $\ell$ defined on $\REuc$.
Let $\shSp{\ell}$ denote the set of all spherical harmonics of exact degree $\ell$ on $\sph{d}$. The dimension of the linear space $\shSp{\ell}$ is
\begin{equation}\label{eq:dim.sph.harmon}
    Z(d,\ell):=(2\ell+d-1)\frac{\Gamma(\ell+d-1)}{\Gamma(d)\Gamma(\ell+1)}\asymp (\ell+1)^{d-1},
\end{equation}
where $\Gamma(\cdot)$ denotes the gamma function, and $a_{\ell}\asymp
b_{\ell}$ means $c\:b_{\ell}\le a_{\ell}\le c' \:b_{\ell}$ for some
positive constants $c$, $c'$, and the asymptotic estimate uses
\cite[Eq.~5.11.12]{NIST:DLMF}. The linear span of
$\shSp{\ell}$, $\ell=0,1,\dots,L$, forms the space
$\sphpo{L}$ of spherical polynomials of degree at most $L$.

Since each pair $\shSp{\ell}$, $\shSp{\ell'}$ for $\ell\neq \ell'\ge0$ is $\mathbb{L}_{2}$-orthogonal, $\sphpo{L}$ is the direct sum of $\shSp{\ell}$, i.e. $\sphpo{L}=\bigoplus_{\ell=0}^{L} \shSp{\ell}$. Moreover, the infinite direct sum $\bigoplus_{\ell=0}^{\infty} \shSp{\ell}$ is dense in $\Lp{p}{d}$, see e.g. \cite[Ch.1]{WaLi2006}. Each member of $\shSp{\ell}$ is an eigenfunction of the negative Laplace-Beltrami operator $-\LBo$ on the sphere $\sph{d}$ with eigenvalue
\begin{equation}\label{eq:eigenvalue}
  \lambda_{\ell}=\lambda_{\ell}^{(d)}:=\ell(\ell+d-1).
\end{equation}

\subsection{Zonal functions}
Let
$\Jcb{\ell}(t)$, $-1\le t\le1$, be the Jacobi polynomial of degree $\ell$ for $\alpha,\beta>-1$. The Jacobi polynomials form an orthogonal polynomial system with respect to the Jacobi weight
$\Jw(t) := (1-t)^{\alpha}(1+t)^{\beta}$, $-1\le t \le 1$.
Given $1\le p\le \infty$ let $\Lpw{p}$ be the $\mathbb{L}_{p}$ space on $[-1,1]$ with respect to the measure $\Jw(t)\IntD{t}$. The space $\Lpw{2}$ forms a Hilbert space with inner product $\InnerLJcb{f,g}:=\int_{-1}^{1}f(t)g(t)\Jw(t)\IntD{t}$ for $f,g\in\Lpw{2}$.
We denote the normalised Legendre or Gegenbauer
polynomial by
\begin{equation*}%\label{eq:NGegen}
  \NGegen{\ell}(t):=\Jcb[\frac{d-2}{2},\frac{d-2}{2}]{\ell}(t)/\Jcb[\frac{d-2}{2},\frac{d-2}{2}]{\ell}(1).
\end{equation*}
A \emph{zonal function} is a function $K:\sph{d}\times\sph{d}\rightarrow \mathbb{R}$ that depends only on the inner product of the arguments, i.e. $K(\PT{x},\PT{y})= \mathfrak{K}(\PT{x}\cdot\PT{y})$,\: $\PT{x},\PT{y}\in \sph{d}$, for some function $\mathfrak{K}:[-1,1]\to \mathbb{R}$.
From \cite[Theorem~7.32.1, p.~168]{Szego1975}, the zonal function $\NGegen{\ell}(\PT{x}\cdot\PT{y})$ is bounded by
\begin{equation*}%\label{eq:NGegen.max}
  \bigl|\NGegen{\ell}(\PT{x}\cdot\PT{y})\bigr|\le 1.
\end{equation*}

Let $\{Y_{\ell,m}=Y_{\ell,m}^{(d)}: \ell\ge0,\; m=1,\dots,Z(d,\ell)\}$ be an \emph{orthonormal basis} for the space $\Lp{2}{d}$.
The normalised Legendre polynomial $\NGegen{\ell}(\PT{x}\cdot\PT{y})$ satisfies the \emph{addition theorem}
\begin{equation}\label{eq:addition.theorem}
  \sum_{m=1}^{Z(d,\ell)}Y_{\ell,m}(\PT{x})Y_{\ell,m}(\PT{y})=Z(d,\ell)\NGegen{\ell}(\PT{x}\cdot \PT{y}).
\end{equation}

\subsection{Generalised Sobolev spaces}
Given $s\in\mathbb{R}_{+}$, let
\begin{equation}\label{eq:b_l}
  \bell{s}:=(1+\lambda_{\ell})^{s/2}\asymp (1+\ell)^{s},
\end{equation}
where $\lambda_{\ell}$ is given by \eqref{eq:eigenvalue}. The \emph{Fourier coefficients} for $f$ in $\Lp{1}{d}$ are
\begin{align}\label{eq:Fo coeff}
 \widehat{f}_{\ell m}:=\int_{\sph{d}}f(\PT{x})\shY(\PT{x}) \IntDiff{x}, \;\; \ell\ge0,\: m=1,\dots,\shd.
\end{align}
The \emph{generalised Sobolev space} $\sob{p}{s}{d}$ may be defined as the set of all functions $f\in \Lp{p}{d}$ satisfying
  $\sum_{\ell=0}^{\infty}\bell{s}\sum_{m=1}^{\shd} \widehat{f}_{\ell m} \shY\: \in\: \Lp{p}{d}$.
The Sobolev space $\sob{p}{s}{d}$ forms a Banach space with norm
\begin{equation}\label{eq:sob.norm.sph}
  \norm{f}{\sob{p}{s}{d}}:=\normB{\sum_{\ell=0}^{\infty}\bell{s}\sum_{m=1}^{\shd} \widehat{f}_{\ell m} \shY}{\Lp{p}{d}}.
\end{equation}
Let $\sob{p}{0}{d}:=\Lp{p}{d}$. The Sobolev norm of $f$ in \eqref{eq:sob.norm.sph} has the equivalent representation
\begin{equation}\label{eq:sob.norm.sph.LBO}
  \norm{f}{\sob{p}{s}{d}}\asymp \normb{\sLB{s/2} f}{\Lp{p}{d}}.
\end{equation}

We have the following two embedding lemmas for $\sob{p}{s}{d}$, see \cite[Section~2.7]{Aubin1998} and \cite{Kamzolov1982}. Let $\ContiSph{}$ be the space of continuous functions on $\sph{d}$.
\begin{lemma}[Continuous embedding into $\ContiSph{}$]\label{lm:embed.sob.C} Let $d\ge2$. The Sobolev space $\sob{p}{s}{d}$ is continuously embedded into $\ContiSph{}$ if $s>d/p$.
\end{lemma}

\begin{lemma}\label{lm:embed.sob.sob} Let $d\ge2$, $s,s'>0$, $1\le p,p'\le\infty$. If $s'-d/p'<s-d/p$, $\sob{p'}{s'}{d}$ is continuously embedded into $\sob{p}{s}{d}$.
\end{lemma}

\subsection{Needlets and filtered operators} We state some known results for needlets and needlet approximations in this section, which we will use later in the paper. All results can be found in \cite{NaPeWa2006-1,NaPeWa2006-2,WaLeSlWo2016}.

As mentioned, spherical needlets are localised polynomials on $\sph{d}$ associated with a quadrature rule and a filter.
A \emph{filtered kernel} on $\sph{d}$ with filter $\fil$ is, for $R\in \mathbb{R}_{+}$,
\begin{equation}\label{eq:filter.sph.ker}
  \vdh{R,\fil}(\PT{x}\cdot\PT{y}):= \vdh[d]{R,\fil}(\PT{x}\cdot\PT{y}) :=\begin{cases}
  1, & 0\le R<1,\\[1mm]
  \displaystyle\sum_{\ell=0}^{\infty}\fil\Bigl(\frac{\ell}{R}\Bigr)\:Z(d,\ell)\:\NGegen{\ell}(\PT{x}\cdot\PT{y}), & R\ge1.
  \end{cases}
\end{equation}

Given $N\ge1$, for $k=1,\dots,N$, let $\PT{x}_{k}$ be $N$ nodes on
$\sph{d}$ and let $w_{k}>0$ be the corresponding weights. The set
$\{(w_{k},\PT{x}_{k}): k=1,\dots,N\}$ is a \emph{positive quadrature
(numerical integration) rule} exact for polynomials of degree up to  $L$
for some $L\ge0$ if
\begin{equation*}
  \int_{\sph{d}}\polysph(\PT{x})\IntDiff{x}=\sum_{k=1}^{N}w_{k}\:\polysph(\PT{x}_{k}),\quad \mbox{for~all~} \polysph\in\sphpo{L}.
\end{equation*}

Let the \emph{needlet filter} $\fiN$ be a filter with specified smoothness $\fis\ge 1$ satisfying
\begin{subequations}\label{subeqs:fiN}
  \begin{align}
  &\fiN\in \CkR, \quad \supp \fiN = [1/2,2];\label{eq:fiN-a}\\[1mm]
  &\fiN(t)^{2} + \fiN(2t)^{2} = 1 \;\hbox{~if~} t\in [1/2,1].\label{eq:fiN-b}
  \end{align}
\end{subequations}
For $j=0,1,\dots$, we define the \emph{(spherical) needlet quadrature}
\begin{equation}\label{eq:QN}
\begin{array}{l}
    \{(\wN,\pN{jk}):k=1,\dots,N_{j}\}, \quad\wN>0, \; k=1,\dots,N_{j},\\
    \hbox{exact for polynomials of degree up to $2^{j+1}-1$}.
\end{array}
\end{equation}
A \emph{(spherical) needlet} $\needlet{jk}$, $k=1,\dots,N_{j}$ of order $j$ with needlet filter $\fiN$ and needlet quadrature \eqref{eq:QN} is then defined by
\begin{subequations}\label{subeqs:needlets}
\begin{equation}\label{eq:needlets-a}
  \needlet{jk}(\PT{x}) :=
  \sqrt{\wN}\: \vdh{2^{j-1},\fiN}(\PT{x}\cdot\PT{x}_{jk}),
\end{equation}
or equivalently, $\needlet{0k}(\PT{x}):=\sqrt{w_{0k}}$,
\begin{align}\label{eq:needlets-b}
  \needlet{jk}(\PT{x}) :=
  \displaystyle\sqrt{\wN}\: \sum_{\ell=0}^{\infty}\fiN\Bigl(\frac{\ell}{2^{j-1}}\Bigr)\:Z(d,\ell)\:\NGegen{\ell}(\PT{x}\cdot\pN{jk}),
  \;\; \hbox{if~} j\ge1.
\end{align}
\end{subequations}
From \eqref{eq:fiN-a} we see that $\needlet{jk}$ is a polynomial of degree $2^{j}-1$. It is a band-limited polynomial, so that $\needlet{jk}$ is $\mathbb{L}_{2}$-orthogonal to all polynomials of degree $\le 2^{j-2}$.
For $f\in \Lp{1}{d}$, the original \emph{(spherical) needlet approximation} with filter $\fiN$ and needlet quadrature \eqref{eq:QN} is defined (see \cite{NaPeWa2006-1}) by
\begin{equation}\label{eq:neapx.f}
  \neapx(f;\PT{x}):=\sum_{j=0}^{\neord}\sum_{k=1}^{N_{j}}\InnerL{f,\needlet{jk}} \needlet{jk}(\PT{x}),\quad \PT{x}\in\sph{d}.
\end{equation}
Thus $\neapx(f;\PT{x})$ is a spherical polynomial of degree at most $2^{\neord}-1$.

As in \cite{WaLeSlWo2016}, we introduce the filter $\fiH$ related to the needlet filter $\fiN$:
\begin{equation}\label{eq:fiH}
\fiH(t):=\left\{\begin{array}{ll}
1, & 0\le t<1,\\
\fiN(t)^{2}, & t\ge1,
\end{array}\right.
\end{equation}
and use the property
\begin{equation}\label{eq:fiH.fiN}
  \fiH\left(\frac{t}{2^{\neord}}\right)=\sum_{j=0}^{\neord}\fiN\left(\frac{t}{2^{j}}\right)^{2}, \quad t\ge1,\;\neord\in\mathbb{Z}_{+},
\end{equation}
which is an easy consequence of \eqref{subeqs:fiN}. We note that \eqref{eq:fiH.fiN} implies $\fiH\in\CkR$ given $\fiN\in\CkR$.

The following theorem shows that an appropriate sum of products of needlets is exactly a filtered kernel. It is proved in \cite[Theorem~3.9]{WaLeSlWo2016} and is already implicit in \cite{NaPeWa2006-1}.
\begin{theorem}[Needlets and filtered kernel]\label{thm:needlets.vs.filter.sph.ker} Let $d\ge2$ and let $\fiN$ be a needlet filter, see \eqref{subeqs:fiN}, and let $\fiH$ be given by \eqref{eq:fiH}. For $j\ge0$ and $1\le k\le N_{j}$, let $\needlet{jk}$ be needlets with filter $\fiN$ and needlet quadrature \eqref{eq:QN}. Then,
\begin{align}
  \sum_{k=1}^{N_{j}}\needlet{jk}(\PT{x})\:\needlet{jk}(\PT{y})
  &=  \vdh{2^{j-1},\fiN^{2}}(\PT{x}\cdot\PT{y}),\;\; j\ge0,\label{eq:needlets.vs.filter.sph.ker.fiN}\\
    \sum_{j=0}^{\neord}\sum_{k=1}^{N_{j}}\needlet{jk}(\PT{x})\:\needlet{jk}(\PT{y})&=\vdh{2^{\neord-1},\fiH}(\PT{x}\cdot\PT{y}),\;\; \neord\ge0.\notag
    %\label{eq:needlets.vs.filter.sph.ker.fiH}
  \end{align}
\end{theorem}

We may define a \emph{filtered approximation} $\Vdh{R,\fil}(f;\cdot)$ on $\Lp{1}{d}$, $R\ge0$ via an integral operator with the filtered kernel $\vdh{R,\fil}(\PT{x}\cdot\PT{y})$ given by \eqref{eq:filter.sph.ker}: for $f\in \Lp{1}{d}$ and $\PT{x}\in\sph{d}$,
\begin{equation}\label{eq:filter.sph.approx}
  \Vdh{R,\fil}(f;\PT{x}):= \Vdh[d]{R,\fil}(f;\PT{x}):= \InnerL{f,\vdh{R,\fil}(\PT{x}\cdot\cdot)}
  =\int_{\sph{d}} f(\PT{y})\:\vdh{R,\fil}(\PT{x}\cdot\PT{y})\:\IntDiff{y}.
\end{equation}
Note that for $R<1$ this is just the integral of $f$.

Theorem~\ref{thm:needlets.vs.filter.sph.ker} with
\eqref{eq:filter.sph.approx} leads to the following equivalence of the
filtered approximation with filter $\fiH$ and the needlet
approximation in \eqref{eq:neapx.f}.
\begin{theorem}\label{thm:needlets.vs.filter.approx} Under the assumptions of Theorem~\ref{thm:needlets.vs.filter.sph.ker}, for $f\in \Lp{1}{d}$ and $\neord\in\Nz$,
\begin{equation*}
  \neapx(f;\PT{x}) =  \int_{\sph{d}} f(\PT{x})\:\vdh{2^{\neord-1},\fiH}(\PT{x}\cdot\PT{y})\:\IntDiff{y} = \fihyper{2^{\neord-1},\fiH}(f;\PT{x}),\quad \PT{x}\in\sph{d}.
\end{equation*}
\end{theorem}

When the filter is sufficiently smooth, the filtered kernel is strongly localised. This is shown in the following theorem proved by Narcowich et al. \cite[Theorem~3.5, p.~584]{NaPeWa2006-2}.
\begin{theorem}[\cite{NaPeWa2006-2}]\label{thm:filter.sph.ker.UB} Let $d\ge2$ and let $\fil$ be a filter in $\CkR$ with $1\leq \fis<\infty$ such that $\fil$ is constant on $[0,a]$ for some $0<a<2$.
Then,
\begin{equation*}%\label{eq:filter.sph.ker.UB}
    \bigl|\vdh{R,\fil}(\cos\theta)\bigr|\leq \frac{c \:R^{d}}{(1 + R \theta)^{\fis}}, \quad R\geq1,
\end{equation*}
where $\cos\theta=\PT{x}\cdot\PT{y}$ for some $\PT{x},\PT{y}\in\sph{d}$ and the constant $c$ depends only on $d$, $\fil$ and $\fis$.
\end{theorem}

When the filter is sufficiently smooth, the $\mathbb{L}_{1}$-norm of the filtered kernel and the operator norm of
the filtered approximation $\Vdh{R,\fil}$ on $\Lp{p}{d}$ are both bounded, see \cite{WaLeSlWo2016}.
\begin{theorem}[\cite{WaLeSlWo2016}]\label{thm:filter.sph.ker.L1norm.UB} Let $d\ge2$ and let $\fil$ be a filter in $\CkR$ with $\fis\ge \floor{\frac{d+3}{2}}$ such that $\fil$ is constant on $[0,a]$ for some $0<a<2$. Then
\begin{equation*}%\label{eq:filter.sph.ker.norm.UB}
    \normb{\vdh{R,\fil}(\PT{x}\cdot\cdot)}{\Lp{1}{d}}\le c_{d,\fil,\fis},\quad \PT{x}\in \sph{d},\;R\ge0.
\end{equation*}
\end{theorem}

\begin{theorem}[\cite{WaSl2016,WaLeSlWo2016}]\label{thm:filter.sph.operator.UB} Let $d\ge2$, $1\leq p\leq \infty$. Let $\fil$ be a filter in $\CkR$ with $\fis\ge \floor{\frac{d+3}{2}}$ such that $\fil$ is constant on $[0,a]$ for some $0<a<2$. Then the filtered approximation $\Vdh{T,\fil}$ on $\Lp{p}{d}$ is an operator of strong type $(p,p)$,
i.e.
\begin{equation*}%\label{eq:filter.sph.approx.norm.UB}
  \normb{\Vdh{R,\fil}}{\mathbb{L}_{p}\to \mathbb{L}_{p}}\le c_{d,\fil,\fis},\quad R\ge0.
\end{equation*}
\end{theorem}

For $L\in\N$, the $\mathbb{L}_{p}$ \emph{error of best approximation} of order $L$ for $f\in \Lp{p}{d}$ is defined by
  $E_{L}(f)_{\Lp{p}{d}}:=\inf_{q\in\sphpo{L}} \norm{f-q}{\Lp{p}{d}}$.
For given $f\in \Lp{1}{d}$ and $p\in[1,\infty]$, $E_{L}(f)_{p}$ is a non-increasing sequence. Since $\sphpo{\ell}$, $\ell=0,1,\dots$, is dense in $\Lp{p}{d}$, the error of best approximation converges to zero as $L\to\infty$, i.e.
\begin{equation}\label{eq:EL.to.0}
    \lim_{L\to\infty}E_{L}(f)_{\Lp{p}{d}}=0, \quad f\in\Lp{p}{d}.
\end{equation}
The error of best approximation for functions in a Sobolev space has the following upper bound, see \cite{Kamzolov1982} and also \cite[p.~1662]{MhNaPrWa2010}.
\begin{lemma}[\cite{Kamzolov1982,MhNaPrWa2010}]\label{lm:best.approx.sph.sob} Let $d\ge2$, $s\ge0$, $1\le p \le \infty$. For $L\in\N$ and $f\in \sob{p}{s}{d}$,
\begin{equation*}%\label{eq:best.approx.err.sob}
  E_{L}(f)_{\Lp{p}{d}}\le c\: L^{-s}\: \norm{f}{\sob{p}{s}{d}},
\end{equation*}
where the constant $c$ depends only on $d$, $p$ and $s$.
\end{lemma}

\begin{theorem}[\cite{WaLeSlWo2016}]\label{thm:needlets.err.L} Let $d\ge2$, $1\le p\le\infty$, $\neord\in\N$. Let $\neapx(f)$, see \eqref{eq:neapx.f}, be the semidiscrete needlet approximation with needlets $\needlet{jk}$, see \eqref{subeqs:needlets}, for filter smoothness $\fis\ge \floor{\frac{d+3}{2}}$. Then for $f\in \Lp{p}{d}$,
\begin{equation*}%\label{eq:needlets.err.L}
  \normb{f-\neapx(f)}{\Lp{p}{d}}\le c \: E_{2^{\neord-1}}(f)_{\Lp{p}{d}},
\end{equation*}
where the constant $c$ depends only on $d$, the filter $\fiN$ and $\fis$.
\end{theorem}

Theorem~\ref{thm:needlets.err.L} and Lemma~\ref{lm:best.approx.sph.sob} imply a rate of convergence of the approximation error of $\neapx(f)$ for $f$ in a Sobolev space, as follows.
\begin{theorem}[\cite{WaLeSlWo2016}]\label{thm:needlets.err.W} Under the assumptions of Theorem~\ref{thm:needlets.err.L}, we have for $f\in \sob{p}{s}{d}$ with $s>0$ and $\neord\in\Nz$,
\begin{equation*}%\label{eq:needlets.err.W}
  \normb{f-\neapx(f)}{\Lp{p}{d}}\le c \: 2^{-\neord s} \norms{f},
\end{equation*}
where the constant $c$ depends only on $d$, $p$, $s$, $\fiN$ and $\fis$.
\end{theorem}

\section{Needlet decomposition for random fields}\label{sec:needlet.decomp.RF}
Let $\RF$ be a $2$-weakly isotropic random field on $\sph{d}$ and let $\{\shY:m=1,\dots,\shd;\ell=0,1,\dots\}$ be an orthonormal spherical harmonic basis for $\Lp{2}{d}$. Then $\RF$ admits an $\mathbb{L}_{2}$ convergent Karhunen-Lo\`{e}ve expansion in terms of $\shY$, see \cite[Theorem~5.13, p.~123]{MaPe2011}, i.e.
\begin{equation}\label{eq:L2 norm T S^2}
  \int_{\sph{d}}|\RF(\PT{x})|^{2}\IntDiff{x}<\infty,\quad \Pas
\end{equation}
and in the $\Lp{2}{d}$ sense, $\Pas$,
\begin{equation}\label{eq:KL.expan.RF}
  \RF \sim \sum_{\ell=0}^{\infty}\sum_{m=1}^{Z(d,\ell)}\InnerL{\RF,\shY}\shY.
\end{equation}

The finiteness of \eqref{eq:L2 norm T S^2} may be generalised to the $\mathbb{L}_{p}$ case. The expansion \eqref{eq:KL.expan.RF} however does not hold for all $\RF$ in $\mathbb{L}_{p}$ when $p\neq 2$. This follows from the fact that, for each $p\neq 2$, there exists an $\mathbb{L}_{p}$-function on $\sph{d}$ such that the spherical harmonic expansion does not converge in the $\mathbb{L}_{p}$ sense, see \cite[Theorem~5.1, p.~248--249]{BoCl1973}.

In contrast, the needlet decomposition converges for all $\Lp{p}{d}$-functions (see Narcowich et al. \cite{NaPeWa2006-2}).
In this section we prove the convergence of the needlet approximation for weakly isotropic random fields on $\sph{d}$.

Let $1\le p\le\infty$ and $\ceil{\infty}:=\infty$. Lemmas~\ref{lm:convg T by needlets-1} and \ref{lm:convg T by needlets-2} below show that a $\ceil{p}$-weakly isotropic random field $\RF$ lies in $\Lppsph{p}{d}$ and admits an expansion in terms of needlets in either $\mathbb{L}_{p}$ or pointwise senses. This result is a generalisation from $2$-weakly isotropic random fields on $\sph{2}$ (see \cite[Theorem~5.13, p.~123]{MaPe2011}) to $n$-weakly isotropic random fields on spheres of arbitrary dimension.
\begin{lemma}[\cite{MaPe2011}]\label{lm:convg T by needlets-1} Let $d\ge2$, $\ordiRF\in\N$. Let $\RF$ be an $\ordiRF$-weakly isotropic random field on $\sph{d}$.
Then $\RF\in \Lp{\ordiRF}{d}$ $\Pas$, i.e.
\begin{align*}\label{eq:convg Lp-a}
  \int_{\sph{d}}|\RF(\PT{x})|^{\ordiRF}\IntDiff{x} < \infty,\quad \Pas .
\end{align*}
\end{lemma}
For completeness we give a proof.
\begin{proof} For $\ordiRF\ge1$, by the Fubini theorem and by the fact that $\RF$ is $\ordiRF$-weakly isotropic,
\begin{equation*}
    \expect{\int_{\sph{d}}|\RF(\PT{x})|^{\ordiRF}\IntDiff{x}}
     =\int_{\sph{d}}\expect{|\RF(\PT{x})|^{\ordiRF}}\IntDiff{x}
     = \int_{\sph{d}}\expect{|\RF(\PT{e}_{d})|^{\ordiRF}}\IntDiff{x} = \expect{|\RF(\PT{e}_{d})|^{\ordiRF}} < \infty,
\end{equation*}
where $\PT{e}_{d}:=(0,\dots,0,1)\in \REuc$ is the north pole of $\sph{d}$.
\end{proof}
We may generalise from the integer $\ordiRF$ in Lemma~\ref{lm:convg T by needlets-1} to any extended real number $p\in [1,\infty]$, as follows.
\begin{lemma}\label{lm:convg T by needlets-2} Let $d\ge2$, $1\le p\le\infty$. Let $\RF$ be a $\ceil{p}$-weakly isotropic random field on $\sph{d}$. Then
  $\RF\in \Lp{p}{d}$ $\Pas$.
\end{lemma}
\begin{remark}
When $\RF$ is a Gaussian random field on $\sph{d}$, the conditions of Lemmas~\ref{lm:convg T by needlets-1} and \ref{lm:convg T by needlets-2} for $2\le p\le \infty$ can be replaced by the condition that $\RF$ is $2$-weakly isotropic, as a consequence of (ii) in Proposition~\ref{prop:isotrRF.weak.strong.Gauss}.
\end{remark}
For completeness we give a proof.
\begin{proof} By definition, for $p=\infty$, since
$\expect{\sup_{\PT{x}\in\sph{d}}|\RF(\PT{x})|}<\infty,$
\begin{equation*}
  \sup_{\PT{x}\in\sph{d}}|\RF(\PT{x})|<\infty,\quad \Pas.
\end{equation*}
For $1\le p<\infty$, Lemma~\ref{lm:convg T by needlets-1} gives
\begin{equation*}
  \RF\in \Lp[,\sigma_{d}]{\ceil{p}}{d}\subset \Lp[,\sigma_{d}]{p}{d},\quad \Pas,
\end{equation*}
thus completing the proof.
\end{proof}

Given $d\ge2$ let $\RF$ be a random field on $\sph{d}$ and let $\needlet{jk}$ be needlets given by \eqref{subeqs:needlets} with needlet filter $\fiN\in \CkR$, $\fis\ge 1$.
For $\neord\in\Nz$ the \emph{semidiscrete needlet approximation} of order $\neord$ for $\RF$ is, by \eqref{eq:neapx.f},
\begin{equation}\label{eq:neapx.RF}
    \neapx(\RF;\omega,\PT{x}) := \sum_{j=0}^{\neord}\sum_{k=1}^{N_{j}}\InnerL{\RF(\omega),\needlet{jk}}\needlet{jk}(\PT{x}), \quad \omega\in \probSp, \:\PT{x}\in\sph{d}.
\end{equation}
Following \cite[Eq.~21]{WaLeSlWo2016}, for $j\ge0$, we can also define the \emph{contribution} of order $j$ to the needlet approximations in \eqref{eq:neapx.RF} as
\begin{equation}\label{eq:Uj.RF}
  \parsum{j}(\RF;\omega,\PT{x}):=\sum_{k=1}^{N_{j}}\InnerL{\RF(\omega),\needlet{jk}}\needlet{jk}(\PT{x}), \quad \omega\in \probSp, \:\PT{x}\in\sph{d},
\end{equation}
so $\neapx(\RF;\omega,\PT{x})=\sum_{j=0}^{\neord}\parsum{j}(\RF;\omega,\PT{x})$.
This with \eqref{eq:needlets.vs.filter.sph.ker.fiN} gives the convolution representation for $\parsum{j}(\RF)$ as
\begin{equation}\label{eq:Uj.RF.conv}
  \parsum{j}(\RF;\omega,\PT{x}):=\InnerL{\RF(\omega,\cdot),\vdh{2^{j-1},\fiN^{2}}(\PT{x}\cdot \cdot)}.
\end{equation}

\subsection{Needlet decomposition on $\Lppsph{p}{d}$}
The isotropy of a random field $\RF$ implies a bound on the $\Lpprob{p}$-norm of $\RF$, which will be used in our proofs.
\begin{theorem}\label{thm:RF.nrm.Lpprob} Let $d\ge2$, $1\le p\le \infty$. Let $\RF$ be a $\ceil{p}$-weakly isotropic random field on $\sph{d}$. Then, for $\PT{x}\in\sph{d}$,
\begin{equation*}
\norm{\RF(\PT{x})}{\Lpprob{p}} \le \norm{\RF(\PT{e}_{d})}{\Lpprob{\ceil{p}}}.
\end{equation*}
In particular, if $p$ is a positive integer,
\begin{align*}
\norm{\RF(\PT{x})}{\Lpprob{p}} = \norm{\RF(\PT{e}_{d})}{\Lpprob{p}}.
\end{align*}
\end{theorem}
\begin{proof}
For $1\le p< \infty$, we let $g(\PT{x}):=|\RF(\PT{x})|^{p}$. Using Jensen's inequality \eqref{eq:Jensen.ineq.probab} then gives, noting $p\le \ceil{p}$,
\begin{equation*}%\label{eq:E.RF.p.ceilp}
 \expect{|\RF(\PT{x})|^{p}}=\expect{g(\PT{x})}
 \le \left\{\expect{\left(g(\PT{x})\right)^{\frac{\ceil{p}}{p}}}\right\}^{\frac{p}{\ceil{p}}}
 =\Bigl\{\expect{|\RF(\PT{x})|^{\ceil{p}}}\Bigr\}^{\frac{p}{\ceil{p}}}
 =\Bigl\{\expect{|\RF(\PT{e}_{d})|^{\ceil{p}}}\Bigr\}^{\frac{p}{\ceil{p}}},
\end{equation*}
where the last equality uses the $\ceil{p}$-weak isotropy of $\RF$.
This gives
\begin{equation*}
\norm{\RF(\PT{x})}{\Lpprob{p}} \le \norm{\RF(\PT{e}_{d})}{\Lpprob{\ceil{p}}}.
\end{equation*}
When $p=\infty$, by Proposition~\ref{prop:isotrRF.weak.strong.Gauss}, $\RF(\PT{x})$ and $\RF(\PT{e}_{d})$ have the same law. Thus $\norm{\RF(\PT{x})}{\Lpprob{\infty}}=\norm{\RF(\PT{e}_{d})}{\Lpprob{\infty}}$.
When $p$ is an integer,
\begin{equation*}%\label{eq:E Uj.T-4}
  \norm{\RF(\PT{x})}{\Lpprob{p}} = \bigl\{\expect{|\RF(\PT{x})|^{p}}\bigr\}^{1/p}
  =\bigl\{\expect{|\RF(\PT{e}_{d})|^{p}}\bigr\}^{1/p} =\norm{\RF(\PT{e}_{d})}{\Lpprob{p}},
\end{equation*}
thus completing the proof.
\end{proof}
The following theorem shows that the needlet approximation for isotropic random fields on $\sph{d}$ converges in $\Lppsph{p}{d}$.
\begin{theorem}[Needlet decomposition on $\Lppsph{p}{d}$]\label{thm:neapx.Lp.converge.RF} Let $d\ge2$, $1\le p\le \infty$. Let $\RF$ be a $\ceil{p}$-weakly isotropic random field on $\sph{d}$. Let $\fiN$ be a needlet filter given by \eqref{subeqs:fiN} satisfying $\fiN\in\CkR$ with $\fis\ge \floor{\frac{d+3}{2}}$. Then, $\RF$ admits the following needlet decomposition:
\begin{equation}\label{eq:T.Uj.expan}
  T\sim \sum_{j=0}^{\infty}\parsum{j}(\RF),
\end{equation}
with the convergence in \eqref{eq:T.Uj.expan} holding in the following sense:\\
(i) for $1\le p<\infty$,
\begin{equation}\label{eq:neapx.Lp.converge.RF}
  \lim_{\neord\to\infty}\expect{\normb{\RF-\neapx(\RF)}{\Lp{p}{d}}^p}=0;
\end{equation}
(ii) for $p=\infty$,
\begin{equation*}%\label{eq:neapx.Linf.converge.RF}
  \lim_{\neord\to\infty}\expect{\normb{\RF-\neapx(\RF)}{\Lp{\infty}{d}}}=0.
\end{equation*}
\end{theorem}
\begin{remark} By the Fubini theorem, \eqref{eq:neapx.Lp.converge.RF} implies that the needlet approximation converges in the $\Lppsph{p}{d}$-norm: let $1\le p< \infty$ and let $\RF$ be $\ceil{p}$-weakly isotropic, then
\begin{equation*}
  \lim_{\neord\to\infty}\normb{\RF-\neapx(\RF)}{\Lppsph{p}{d}}=0.
\end{equation*}
\end{remark}

\begin{proof} The proof uses Lebesgue's dominated convergence theorem.
For all $1\le p\le \infty$, by Lemma~\ref{lm:convg T by needlets-2}, $\RF\in \Lp[,\sigma_{d}]{p}{d}$ $\Pas$. For $p<\infty$, Theorem~\ref{thm:needlets.err.L} with \eqref{eq:EL.to.0} implies the almost sure convergence to zero of the $\Lp{p}{d}$-norms of needlet approximation errors:
\begin{equation}\label{eq:pnt convg T-1}
  \lim_{\neord\to\infty}\int_{\sph{d}}\bigl|\RF(\omega,\PT{x})
  -\neapx\bigl(\RF;\omega,\PT{x}\bigr)\bigr|^{p}\IntDiff{x}=0, \quad \Pas.
\end{equation}
Using Theorem~\ref{thm:needlets.err.L} again and Lemma~\ref{lm:best.approx.sph.sob} with $s=0$ gives, $\Pas$,
\begin{equation*}
  \int_{\sph{d}}\bigl|\RF(\omega,\PT{x})-\neapx\bigl(\RF;\omega,\PT{x}\bigr)\bigr|^{p}\IntDiff{x}
   \le c_{d,p,\fiN,\fis}\: E_{\floor{2^{\neord-1}}}(\RF(\omega))_{\Lp{p}{d}}^{p}
   \le c_{d,p,\fiN,\fis}\:\norm{\RF(\omega)}{\Lp{p}{d}}^{p}.
\end{equation*}
Taking expected values on both sides with the Fubini theorem and Theorem~\ref{thm:RF.nrm.Lpprob} gives
\begin{align*}%\label{eq:pnt convg T-2}
  \expect{\int_{\sph{d}}\bigl|\RF(\PT{x})-\neapx\bigl(\RF;\PT{x}\bigr)\bigr|^{p}\IntDiff{x}}
  &\le c_{d,p,\fiN,\fis} \:\expect{\int_{\sph{d}}|\RF(\PT{x})|^{p}\IntDiff{x}}\notag\\
  &= c_{d,p,\fiN,\fis} \:\int_{\sph{d}}\expect{|\RF(\PT{x})|^{p}}\IntDiff{x}\notag\\
  &\le c_{d,p,\fiN,\fis}\:\Bigl\{\expect{|\RF(\PT{e}_{d})|^{\ceil{p}}}\Bigr\}^{\frac{p}{\ceil{p}}}
  < \infty.
\end{align*}
This with \eqref{eq:pnt convg T-1} and Lebesgue's dominated convergence theorem (with respect to the probability measure $\probm$) proves \eqref{eq:neapx.Lp.converge.RF}.

For $p=\infty$, Theorem~\ref{thm:needlets.err.L} with \eqref{eq:EL.to.0} gives, for $\omega\in \probSp$,
\begin{align}\label{eq:pnt convg T inf-1}
    \normb{\RF(\omega)-\neapx(\RF;\omega)}{\Lp{\infty}{d}}\le c_{d,\fiN,\fis}\: E_{\floor{2^{\neord-1}}}(\RF(\omega))_{\Lp{\infty}{d}}\to0,
    \;\; \neord\to\infty,
\end{align}
since $\RF(\omega)$ is $\Pas$ continuous.
Also, by Lemma~\ref{lm:best.approx.sph.sob} with $s=0$,
\begin{align}\label{eq:pnt convg T inf-2}
    \expect{E_{\floor{2^{\neord-1}}}(\RF)_{\Lp{\infty}{d}}}\le c_{d}\: \expect{\norm{\RF}{\Lp{\infty}{d}}} < \infty.
\end{align}
Using Lebesgue's dominated convergence theorem again with \eqref{eq:pnt convg T inf-1} and \eqref{eq:pnt convg T inf-2} gives
\begin{align*}
    \expect{\normb{\RF-\neapx(\RF)}{\Lp{\infty}{d}}}\to0,
    \;\; \neord\to\infty,
\end{align*}
thus completing the proof.
\end{proof}

\subsection{Pointwise convergence}
The following theorem shows the pointwise convergence of the needlet decomposition for weakly isotropic random fields of even order. For $1\le p< \infty$, let $\ceil[2]{p}$ be the smallest even integer larger than or equal to $p$,
\begin{equation}\label{eq:p.ceil2}
  \ceil[2]{p} := \left\{\begin{array}{ll}
  \ceil{p}, & \ceil{p}~ \hbox{is~even},\\
  \ceil{p}+1, & \ceil{p}~ \hbox{is~odd}.
  \end{array}\right.
\end{equation}
\begin{theorem}\label{thm:converg needlet pointwise} Let $d\ge2$, $1\le p< \infty$. Let $\RF$ be a $\ceil[2]{p}$-weakly isotropic random field on $\sph{d}$. Let $\fiN$ be a needlet filter satisfying $\fiN\in \CkR$ with $\fis\ge \floor{\frac{d+3}{2}}$. Then, for each $\PT{x}\in\sph{d}$,
\begin{equation}\label{eq:converg needlet pointwise-1}
  \lim_{\neord\to\infty}\expect{\bigl|\RF(\PT{x})-\neapx(\RF;\PT{x})\bigr|^p}=0.
\end{equation}
\end{theorem}

Before we can prove Theorem~\ref{thm:converg needlet pointwise}, we first need the results of Lemma~\ref{lm:Uj.RF.rot} and Theorem~\ref{thm:T.neapx.isotr} below. We first show that the weak isotropy of $\RF$ extends to $\neapx(\RF)$. By Theorem~\ref{thm:needlets.vs.filter.approx}, the needlet approximation $\neapx(\RF)$ has the convolution representation
\begin{equation}\label{eq:neapx.RF.conv}
  \neapx(\RF;\PT{x}) = \InnerL{\RF, \vdh{2^{\neord-1},\fiH}(\PT{x}\cdot\cdot)},
\end{equation}
where $\vdh{2^{\neord-1},\fiH}(\PT{x}\cdot\PT{y})$ is the filtered kernel with the filter $\fiH$, see \eqref{eq:fiH}.
For any rotation $\rho\in \RotGr$,
\begin{align*}%\label{eq:neapx.RF.rot.invariant}
  \neapx(\RF;\rho\PT{x}) &= \int_{\sph{d}}\RF(\PT{y})\:\vdh{2^{\neord-1},\fiH}(\rho\PT{x}\cdot\PT{y})\IntDiff{y}\notag\\
  &= \int_{\sph{d}}\RF(\PT{y})\:\vdh{2^{\neord-1},\fiH}(\PT{x}\cdot\rho^{-1}\PT{y})\IntDiff{y}\notag\\
  &= \int_{\sph{d}}\RF(\rho\PT{y})\:\vdh{2^{\neord-1},\fiH}(\PT{x}\cdot\PT{y})\IntDiff{y}\notag\\
  &= \InnerLb{\RF(\rho\cdot), \vdh{2^{\neord-1},\fiH}(\PT{x}\cdot\cdot)}\notag\\
  &= \neapx(\RF(\rho\cdot);\PT{x}),
\end{align*}
where the third equality exploits the rotational invariance of the integral over $\sph{d}$.
Thus we have the following formula for $\neapx(\RF)$ under rotations.
\begin{lemma}\label{lm:Uj.RF.rot} Let $d\ge2$, $\neord\in\Nz$. Let $\RF$ be a random field on $\sph{d}$. Let $\neapx(\RF)$ be needlet approximations for $\RF$ in \eqref{eq:neapx.RF} with filter $\fiN\in \CkR$, $\kappa\ge 1$. Then for $\rho\in\RotGr$,
\begin{equation*}%\label{eq:neapx.RF.rot.invariant-2}
  \neapx(\RF;\rho\PT{x})=\neapx(\RF(\rho\cdot);\PT{x}).
\end{equation*}
\end{lemma}

\begin{theorem}\label{thm:T.neapx.isotr} Let $d\ge2$, $\ordiRF\in\N$. Let $\RF$ be an $\ordiRF$-weakly isotropic random field on $\sph{d}$. Let $\neapx(\RF)$ be the semidiscrete needlet approximation of order $\neord$ for $\RF$, see \eqref{eq:neapx.f}. Then $\{\RF,\neapx(\RF): \neord\in \Nz\}$ is $\ordiRF$-weakly isotropic.
\end{theorem}

\begin{proof} First, we show
\begin{equation}\label{eq:E Uj.T-1}
  \expect{|\neapx(\RF;\PT{x})|^{\ordiRF}}<\infty, \quad \mbox{for~}\PT{x}\in\sph{d}.
\end{equation}
For $\ordiRF\ge1$, by \eqref{eq:neapx.RF.conv} and Minkowski's inequality for integrals, see e.g. \cite[Appendix~A]{Stein1970},
\begin{align}\label{eq:E Uj.T-2}
  \left\{\expect{|\neapx(\RF;\PT{x})|^{\ordiRF}}\right\}^{1/\ordiRF}
  &= \normB{\int_{\sph{d}}\RF(\PT{y})\vdh{2^{\neord-1},\fiH}(\PT{x}\cdot\PT{y})\IntDiff{y}}{\Lpprob{\ordiRF}}\notag\\
  &\le \int_{\sph{d}}\norm{\RF(\PT{y})}{\Lpprob{\ordiRF}}|\vdh{2^{\neord-1},\fiH}(\PT{x}\cdot\PT{y})|\IntDiff{y}.
\end{align}
Since $\RF$ is $n$-weakly isotropic, the second part of Theorem~\ref{thm:RF.nrm.Lpprob} then gives
\begin{align}\label{eq:E Uj.T-4}
  \norm{\RF(\PT{y})}{\Lpprob{\ordiRF}} =\norm{\RF(\PT{e}_{d})}{\Lpprob{\ordiRF}}<\infty,
\end{align}
where $\PT{e}_{d}$ is the north pole of $\sph{d}$. Also, by Theorem~\ref{thm:filter.sph.ker.L1norm.UB},
\begin{equation*}
    \normb{\vdh{2^{\neord-1},\fiH}(\PT{x}\cdot\cdot)}{\Lp{1}{d}}\le c_{d,\fiN,\fis}.
\end{equation*}
This together with \eqref{eq:E Uj.T-2} and \eqref{eq:E Uj.T-4} gives
\begin{equation*}%\label{eq:E Uj.T-3}
  \left\{\expect{|\neapx(\RF;\PT{x})|^{\ordiRF}}\right\}^{1/\ordiRF}
   \le\norm{\RF(\PT{e}_{d})}{\Lpprob{\ordiRF}}\normb{\vdh{2^{\neord-1},\fiH}(\PT{x}\cdot\cdot)}{\Lp{1}{d}}
   \le c_{d,\fiN,\fis}\:\norm{\RF(\PT{e}_{d})}{\Lpprob{\ordiRF}} <\infty,
\end{equation*}
thus proving \eqref{eq:E Uj.T-1}.

We can now prove the weak isotropy of the set $\{\RF,\neapx(\RF) : \neord\in\Nz\}$. For $1\le k\le \ordiRF$, let $\neord_{1},\dots,\neord_{k}$ be $k$ positive integers and $\rho$ be a rotation on $\sph{d}$. We need to show the following two identities: for $\PT{x},\PT{x}_{1},\dots,\PT{x}_{k}\in\sph{d}$,
\begin{subequations}\label{eq:w.Isotropy T Uj.T}
  \begin{align}
    \expect{\RF(\rho\PT{x})\neapx[{\neord_{1}}](\RF;\rho\PT{x}_{1})\cdots \neapx[{\neord_{k-1}}](\RF;\rho\PT{x}_{k-1})}
    &=\expect{\RF(\PT{x})\neapx[{\neord_{1}}](\RF;\PT{x}_{1})\cdots \neapx[{\neord_{k-1}}](\RF;\PT{x}_{k-1})}\label{eq:w.Isotropy T Uj.T-a}\\[2mm]
    \expect{\neapx[{\neord_{1}}](\RF;\rho\PT{x}_{1})\cdots \neapx[{\neord_{k}}](\RF;\rho\PT{x}_{k})}&=\expect{\neapx[{\neord_{1}}](\RF;\PT{x}_{1})\cdots \neapx[{\neord_{k}}](\RF;\PT{x}_{k})}.
    \label{eq:w.Isotropy T Uj.T-b}
  \end{align}
\end{subequations}
For \eqref{eq:w.Isotropy T Uj.T-a}, by Lemma~\ref{lm:Uj.RF.rot} and \eqref{eq:neapx.RF.conv},
\begin{align*}
    &\hspace{0.55cm}\expect{\RF(\rho\PT{x})\neapx[{\neord_{1}}](\RF;\rho\PT{x}_{1})\cdots \neapx[{\neord_{k-1}}](\RF;\rho\PT{x}_{k-1})}\\
        &=\expect{\RF(\rho\PT{x})\neapx[{\neord_{1}}](\RF(\rho\cdot);\PT{x}_{1})\cdots \neapx[{\neord_{k-1}}](\RF(\rho\cdot);\PT{x}_{k-1})}\\
        & = \expect{\underbrace{\int_{\sph{d}}\cdots\int_{\sph{d}}}_{k-1}\RF(\rho\PT{x})\prod_{i=1}^{k-1}\left(\RF(\rho\PT{z}_{i})\vdh{2^{\neord_{i}-1},\fiH}(\PT{x}_{i}\cdot\PT{z}_{i})\right)\IntDiffs{\PT{z}_{1}}\cdots\IntDiffs{\PT{z}_{k-1}}}\\
        &=\int_{\sph{d}}\cdots\int_{\sph{d}}\expect{\RF(\rho\PT{x})\RF(\rho\PT{z}_{1})\cdots \RF(\rho\PT{z}_{k-1})}\prod_{i=1}^{k-1}\left(\vdh{2^{\neord_{i}-1},\fiH}(\PT{x}_{i}\cdot\PT{z}_{i})\right)\IntDiffs{\PT{z}_{1}}\cdots\IntDiffs{\PT{z}_{k-1}}\\
        & = \int_{\sph{d}}\cdots\int_{\sph{d}}\expect{\RF(\PT{x})\RF(\PT{z}_{1})\cdots \RF(\PT{z}_{k-1})}\prod_{i=1}^{k-1}\left(\vdh{2^{\neord_{i}-1},\fiH}(\PT{x}_{i}\cdot\PT{z}_{i})\right)\IntDiffs{\PT{z}_{1}}\cdots\IntDiffs{\PT{z}_{k-1}}\\
        & = \expect{\RF(\PT{x})\neapx[{\neord_{1}}](\RF;\PT{x}_{1})\cdots \neapx[{\neord_{k-1}}](\RF;\PT{x}_{k-1})},
  \end{align*}
where the third and last equalities use the Fubini theorem and the fourth equality uses the $\ordiRF$-weak isotropy of $\RF$.
The proof of \eqref{eq:w.Isotropy T Uj.T-b} follows by a similar argument.
\end{proof}

Using a similar argument to the proof for Theorem~\ref{thm:T.neapx.isotr} with the integral representation in \eqref{eq:Uj.RF.conv} for the order-$j$ contribution $\parsum{j}(\RF)$,
we can prove the weakly isotropic property for the set $\{\RF,\parsum{j}(\RF): j=1,2,\dots\}$, as stated in the following theorem.

\begin{theorem}\label{thm:T.parsum.isotr} Let $d\ge2$, $\ordiRF\in\N$. Let $\RF$ be an $\ordiRF$-weakly isotropic random field on $\sph{d}$. Let $\parsum{j}(\RF)$ be the order-$j$ contribution in \eqref{eq:Uj.RF} for $\neapx(\RF)$. Then $\{\RF,\parsum{j}(\RF):j\in\Nz\}$ is $\ordiRF$-weakly isotropic.
\end{theorem}

\begin{proof}[Proof of Theorem~\ref{thm:converg needlet pointwise}]
To prove \eqref{eq:converg needlet pointwise-1}, we let $f_{\neord}(\PT{x}):=|\RF(\PT{x})-\neapx(\RF;\PT{x})|^{p}$. The fact that $\ceil[2]{p}$ is even, see \eqref{eq:p.ceil2}, together with Theorem~\ref{thm:T.neapx.isotr} gives, for each point $\PT{x}\in \sph{d}$ and each rotation $\rho\in \RotGr$,
\begin{align}\label{eq:T Uj isotropy}
    \expect{\bigl|\RF(\rho\PT{x})-\neapx(\RF;\rho\PT{x})\bigr|^{\ceil[2]{p}}}=\expect{\bigl|\RF(\PT{x})-\neapx(\RF;\PT{x})\bigr|^{\ceil[2]{p}}}.
\end{align}
Then for $\PT{x}\in\sph{d}$, using Jensen's inequality \eqref{eq:Jensen.ineq.probab} and $p\le \ceil[2]{p}$,
\begin{align}\label{eq:converg needlet pnt-1}
  \expect{\bigl|\RF(\PT{x})-\neapx(\RF;\PT{x})\bigr|^{p}}
  & =\expect{|f_{\neord}(\PT{x})|}\le \left\{\expect{|f_{\neord}(\PT{x})|^{\frac{\ceil[2]{p}}{p}}}\right\}^{\frac{p}{\ceil[2]{p}}}\notag\\[0.2cm]
  & = \left\{\expect{\bigl|\RF(\PT{x})-\neapx(\RF;\PT{x})\bigr|^{\ceil[2]{p}}}\right\}^{\frac{p}{\ceil[2]{p}}}\notag\\
  & = \left\{\int_{\sph{d}}\expect{\bigl|\RF(\PT{y})-\neapx(\RF;\PT{y})\bigr|^{\ceil[2]{p}}}\IntDiff{y}\right\}^{\frac{p}{\ceil[2]{p}}}\notag\\
  & = \left\{\expect{\int_{\sph{d}}\bigl|\RF(\PT{y})-\neapx(\RF;\PT{y})\bigr|^{\ceil[2]{p}}\IntDiff{y}}\right\}^{\frac{p}{\ceil[2]{p}}},
\end{align}
where the third equality uses \eqref{eq:T Uj isotropy} and the fact that $\sigma_{d}$ is normalised, and the last equality uses the Fubini theorem. By Theorem~\ref{thm:neapx.Lp.converge.RF}, the last formula of \eqref{eq:converg needlet pnt-1} converges to zero. This then gives
\begin{align*}
  \expect{\bigl|\RF(\PT{x})-\neapx(\RF;\PT{x})\bigr|^{p}}\to 0,\;\; \neord\to \infty,
\end{align*}
completing the proof.
\end{proof}

\subsection{Boundedness of needlet approximation}\label{subsec:Bd.neapx.Lppsph}
In this subsection we prove that the $\Lppsph{p}{d}$-norm, $1\le p<\infty$, of $\neapx(\RF)$ is bounded by the $\Lppsph{\ceil{p}}{d}$-norm of the random field $\RF$ when $\RF$ is $\ceil{p}$-weakly isotropic and the filter $\fiN$ is sufficiently smooth. We state the result in the following theorem.
\begin{theorem}[Boundedness of semidiscrete needlet approximation]\label{thm:Bd.neapx.Lppsph} Let $d\ge 2$, $1\le p< \infty$. Let $\RF$ be a $\ceil{p}$-weakly isotropic random field on $\sph{d}$. Let $\fiN$ be a needlet filter in \eqref{subeqs:fiN} satisfying $\fiN\in\CkR$ and $\fis\ge \floor{\frac{d+3}{2}}$. Then
\begin{equation}\label{eq:norm.Lppsph.neapx.RF}
    \normb{\neapx(\RF)}{\Lppsph{p}{d}}\le c\:\norm{\RF}{\Lppsph{\ceil{p}}{d}},
\end{equation}
where the constant $c$ depends only on $d$, $\fiN$ and $\fis$.
\end{theorem}
We need the following lemma to prove Theorem~\ref{thm:Bd.neapx.Lppsph}.
\begin{lemma}\label{lm:RF.Lpp.Lppsph.nrms} Let $d\ge2$, $\ordiRF\in\N$. Let $\RF$ be an $\ordiRF$-weakly isotropic random field on $\sph{d}$. Then
\begin{align*}%\label{eq:RF.Lpp.Lppsph.nrms}
    \norm{\RF(\PT{e}_{d})}{\Lpprob{\ordiRF}}
      =  \norm{\RF}{\Lppsph{\ordiRF}{d}}.
\end{align*}
\end{lemma}
\begin{proof} The $\ordiRF$-weak isotropy of $\RF$ and the Fubini theorem give
\begin{align*}%\label{eq:RF.Lpp.Lppsph.nrms}
    \norm{\RF(\PT{e}_{d})}{\Lpprob{\ordiRF}}
      = \left\{\expect{\bigl|\RF(\PT{e}_{d})\bigr|^{\ordiRF}}\right\}^{\frac{1}{\ordiRF}}
      = \left\{\int_{\sph{d}}\expect{\bigl|\RF(\PT{x})\bigr|^{\ordiRF}}\IntDiff{x}\right\}^{\frac{1}{\ordiRF}}
      = \norm{\RF}{\Lppsph{\ordiRF}{d}},
\end{align*}
completing the proof.
\end{proof}
\begin{proof}[Proof of Theorem~\ref{thm:Bd.neapx.Lppsph}]
For $1\le p<\infty$, by the Fubini theorem and \eqref{eq:neapx.RF.conv},
\begin{align*}%\label{eq:RF.nrm.Lppsph}
  \normb{\neapx(\RF)}{\Lppsph{p}{d}}^{p}
  &= \int_{\sph{d}}\int_{\probSp}|\neapx(\RF;\omega,\PT{x})|^{p}\:\Dpb\IntDiff{x}\\
  &= \int_{\sph{d}}\normB{\int_{\sph{d}}\RF(\PT{y})\vdh{2^{\neord-1},\fiH}(\PT{x}\cdot\PT{y})\IntDiff{y}}{\Lpprob{p}}^{p}\IntDiff{x}\notag\\
  &\le \int_{\sph{d}}\left|\int_{\sph{d}}\norm{\RF(\PT{y})}{\Lpprob{p}}|\vdh{2^{\neord-1},\fiH}(\PT{x}\cdot\PT{y})|\IntDiff{y}\right|^{p}\IntDiff{x},
\end{align*}
where the inequality uses the Minkowski's inequality for integrals, see e.g. \cite[Appendix~A]{Stein1970}.
By Theorem~\ref{thm:RF.nrm.Lpprob},
\begin{equation*}
\norm{\RF(\PT{y})}{\Lpprob{p}}
  \le \norm{\RF(\PT{e}_{d})}{\Lpprob{\ceil{p}}},
\end{equation*}
where $\PT{e}_{d}$ is the north pole of $\sph{d}$. This with Lemma~\ref{thm:filter.sph.ker.L1norm.UB} then gives
\begin{align*}%\label{eq:RF.nrm.Lppsph}
  \normb{\neapx(\RF)}{\Lppsph{p}{d}}^{p}
  \le \norm{\RF(\PT{e}_{d})}{\Lpprob{\ceil{p}}}^{p}\int_{\sph{d}}\normb{\vdh{2^{\neord-1},\fiH}(\PT{x}\cdot\cdot)}{\Lp{1}{d}}^{p}\IntDiff{x}
  \le c_{d,\fiN,\fis}^{p}\:\norm{\RF(\PT{e}_{d})}{\Lpprob{\ceil{p}}}^{p}.
\end{align*}
This with Lemma~\ref{lm:RF.Lpp.Lppsph.nrms} gives \eqref{eq:norm.Lppsph.neapx.RF}.
\end{proof}

\section{Approximation errors for smooth random fields}\label{sec:truncate err}

This section studies approximation errors for needlet approximations. We prove that for a $2$-weakly isotropic random field $\RF$ on $\sph{d}$, the order of convergence of the approximation error depends upon the rate of decay of the tail of the \emph{angular power spectrum}.

\subsection{Angular power spectrum}\label{subsec:angular.power.spectrum}

Let $\RF$ be a $2$-weakly isotropic random field on $\sph{d}$. The \emph{centered random field} corresponding to $\RF$ is
\begin{equation}\label{eq:RFc}
    \RFc(\omega,\PT{x}) := \RF(\omega,\PT{x}) - \expect{\RF(\PT{x})}.
\end{equation}
Because $\RF$ is $2$-weakly isotropic, the expected value $\expect{\RF(\PT{x})}$ is independent of $\PT{x}$.
The covariance $\expect{\RFc(\PT{x})\RFc(\PT{y})}$, by virtue of its rotational invariance,
is a zonal kernel on $\sph{d}$
\begin{equation}\label{eq:covariance.G}
  \covarRF(\PT{x}\cdot\PT{y}):=\expect{\RFc(\PT{x})\RFc(\PT{y})}.
\end{equation}
This zonal function $\covarRF(\cdot)$ is said to be the \emph{covariance function} for $\RF$.
Given $d\ge2$, let $\alpha:=(d-2)/2$.
When $\covarRF(\cdot)$ is in $\Lpw[{\Jw[\alpha,\alpha]}]{2}$ it has a convergent Fourier expansion
\begin{equation*}
  \covarRF \sim \sum_{\ell=0}^{\infty}\APS[(d)]{\ell}\shd \NGegen{\ell},
\end{equation*}
where the convergence is in the $\Lpw[{\Jw[\alpha,\alpha]}]{2}$ sense.
The set of Fourier coefficients
\begin{equation*}%\label{eq:APS.def}
    \APS{\ell}:=\APS[(d)]{\ell}
    := \int_{\sph{d}}\covarRF(\PT{x}\cdot\PT{y}) \NGegen{\ell}(\PT{x}\cdot\PT{y}) \IntDiff{x}
\end{equation*}
is said to be the \emph{angular power spectrum} for the random field $\RF$. Using the property of zonal functions,
\begin{equation*}%\label{eq:APS.int.one.dim}
    \APS{\ell} = \frac{|\sph{d-1}|}{|\sph{d}|} \int_{-1}^{1} \covarRF(t)\NGegen{\ell}(t)(1-t^{2})^{\frac{d-2}{2}} \IntD{t}.
\end{equation*}
By the addition theorem \eqref{eq:addition.theorem} we can write
\begin{equation}\label{eq:covariance.G.expan}
  \expect{\RFc(\PT{x})\RFc(\PT{y})}=\covarRF(\PT{x}\cdot\PT{y})\sim \sum_{\ell=0}^{\infty} \APS{\ell} \shd \NGegen{\ell}(\PT{x}\cdot\PT{y})=\sum_{\ell=0}^{\infty} \APS{\ell} \sum_{m=1}^{\shd} \shY(\PT{x})\shY(\PT{y}).
\end{equation}
As a function of $\PT{x}\in \sph{d}$, the expansion in \eqref{eq:covariance.G.expan} converges in $\Lp{2}{d}$. The orthogonality of $\shY$ then gives
\begin{equation}\label{eq:covar.G.proj}
  \InnerLb{\covarRF(\cdot\cdot\PT{y}),\shY(\cdot)} = \APS{\ell} \shY(\PT{y}).
\end{equation}

We define \emph{Fourier coefficients} for a random field $\RF$ by, cf. \eqref{eq:Fo coeff},
\begin{equation}\label{eq:Fcoe.RF}
  \widehat{\RF}_{\ell m} := \InnerL{\RF, \shY},\quad \ell\ge0,\; 1\le m\le \shd.
\end{equation}

The formula in \eqref{eq:covar.G.proj} implies the following ``orthogonality'' of Fourier coefficients $\RFcFcoe$ for $\RFc$.
It is a trivial generalisation of \cite[p.~125]{MaPe2011} from $\sph{2}$ to higher dimensional spheres.
\begin{lemma}\label{lm:orth.Fcoe.RF} For $d\ge2$, let $\alpha:=(d-2)/2$. Let $\RF$ be a $2$-weakly isotropic random field on $\sph{d}$ and let $\covarRF(\PT{x}\cdot\PT{y})$ be the covariance function for $\RF$ in \eqref{eq:covariance.G} satisfying $\covarRF(\cdot)\in \Lpw[{\Jw[\alpha,\alpha]}]{1}$. Then for $\ell,\ell'\ge0$ and $m,m'=1,\dots,\shd$,
\begin{equation}\label{eq:orth.Fcoe.RF}
  \expect{\RFcFcoe\RFcFcoe[\ell' m']}=\APS{\ell}\delta_{\ell \ell'} \delta_{m m'},
\end{equation}
where $\RFc$ is given by \eqref{eq:RFc} and $\delta_{\ell\ell'}$ is the Kronecker delta.
\end{lemma}
For completeness we give a proof.
\begin{proof} By \eqref{eq:Fcoe.RF} and \eqref{eq:covar.G.proj} and the $\mathbb{L}_{2}$-orthogonality of the $\shY$,
\begin{align*}
  \expect{\RFcFcoe\RFcFcoe[\ell' m']}
  & = \expect{\int_{\sph{d}}\RFc(\PT{x})Y_{\ell,m}(\PT{x}) \IntDiffs{\PT{x}}\int_{\sph{d}}\RFc(\PT{x}')Y_{\ell',m'}(\PT{x}')\IntDiffs{\PT{x}'}}\\
  & = \int_{\sph{d}}\left(\int_{\sph{d}}\expect{\RFc(\PT{x})\RFc(\PT{x}')}Y_{\ell,m}(\PT{x})\IntDiffs{\PT{x}}\right)Y_{\ell',m'}(\PT{x}')\IntDiffs{\PT{x}'}\\
  & = \int_{\sph{d}}\APS{\ell} Y_{\ell,m}(\PT{x}')Y_{\ell',m'}(\PT{x}')\IntDiffs{\PT{x}'}\\
  & = \APS{\ell} \delta_{\ell m}\delta_{\ell' m'},
\end{align*}
where the second equality uses the Fubini theorem, thus proving \eqref{eq:orth.Fcoe.RF}.
\end{proof}

\begin{remark}
Lemma~\ref{lm:orth.Fcoe.RF} implies that the $\APS{\ell}$ are non-negative.
\end{remark}

The angular power spectrum contains the full information of the covariance of a random field. It plays an important role in depicting cosmic microwave background (CMB) and environmental data in astrophysics and geoscience, see e.g. \cite{CaMa2009,LaDu_etal2014,MaPe2011,HuZhRo2011,JeMaWaYa2014}.

Let $\ellp{2}$ be the collection of all real sequences $(x_{j}:j=1,2,\dots)$ satisfying
$\sum_{j=1}^{\infty}x_{j}^{2}<\infty$.
The following theorem, proved by Lang and Schwab \cite[Section~3]{LaSc2015}, provides a necessary and sufficient condition in terms of $\APS{\ell}$ for $\Diff{s}\covarRF(t)$ to be in $\Lpw[{\Jw[s+\alpha,s+\alpha]}]{2}$ with $s\ge0$ and $\alpha:=(d-2)/2$.
\begin{theorem}[\cite{LaSc2015}]\label{thm:APS.G.in.L2}
  Let $d\ge2$, $s\ge0$ and $\alpha:=(d-2)/2$. Let $\RF$ be a $2$-weakly isotropic random field on $\sph{d}$ and the covariance function $\covarRF(\cdot)$ be given by \eqref{eq:covariance.G}. Then $\Diff{s}\covarRF(t)$ is in $\Lpw[{\Jw[s+\alpha,s+\alpha]}]{2}$
 if and only if the sequence $(\ell^{s+\frac{d-1}{2}}\APS{\ell},\ell\ge 1)$ is in $\ellp{2}$, i.e.
   $\sum_{\ell=1}^{\infty} \bigl|\APS{\ell}\bigr|^{2}\ell^{2s+d-1} < \infty$.
\end{theorem}

\begin{theorem}\label{thm:APS.RFc.in.sob}
 Let $d\ge2$, $s> d/2$. Let $\RF$ be a $2$-weakly isotropic random field on $\sph{d}$ with the angular power spectrum $\APS{\ell}$ satisfying
 \begin{equation}\label{eq:APS.series.bounded.sob}
  \sum_{\ell=1}^{\infty}\;\APS{\ell} \ell^{2s+d-1} < \infty.
 \end{equation}
 Let $\RFc$ be given by \eqref{eq:RFc}. Then
   $\RFc \in \sob{2}{s}{d}\; \Pas$,
and
\begin{equation}\label{eq:E.RF.Sob.APS}
 \expect{\norm{\RFc}{\sob{2}{s}{d}}^{2}} = \sum_{\ell=0}^{\infty} \bell{2s} \shd \APS{\ell}<\infty,
\end{equation}
where $\bell{2s}$ and $\shd$ are given by \eqref{eq:b_l} and \eqref{eq:dim.sph.harmon} respectively.
\end{theorem}

\begin{proof} The first part of Theorem~\ref{thm:APS.RFc.in.sob} was proved in \cite[Section~4]{LaSc2015}. For completeness, we prove both parts.
By \eqref{eq:APS.series.bounded.sob},
\begin{equation*}%\label{eq:angular vs T smoothness-4}
 \sum_{\ell=0}^{\infty} \bigl|\APS{\ell}\bigr|^{2} \ell^{d-1} \le  \left(\sum_{\ell=0}^{\infty} \APS{\ell} \ell^{\frac{d-1}{2}}\right)^{2}
 \le  \left(\sum_{\ell=0}^{\infty} \APS{\ell} \ell^{2s+d-1}\right)^{2} < \infty.
\end{equation*}
This and Theorem~\ref{thm:APS.G.in.L2} with $s=0$ show that the covariance function $\covarRF(\cdot)$ for $\RF$ is in $\Lpw[{\Jw[\alpha,\alpha]}]{2}\subset \Lpw[{\Jw[\alpha,\alpha]}]{1}$, where $\alpha:=(d-2)/2$.
Lemma~\ref{lm:orth.Fcoe.RF} then gives
 \begin{equation}\label{eq:angular-1}
  \expect{|\RFcFcoe|^{2}} = \APS{\ell},\quad 1\le m\le \shd.
 \end{equation}
This with \eqref{eq:APS.series.bounded.sob} gives
\begin{equation*}%\label{eq:angular vs T smoothness-3}
  \expect{\sum_{\ell=0}^{\infty}\sum_{m=1}^{\shd} \bell{2s}|\RFcFcoe|^{2}}
   =  \sum_{\ell=0}^{\infty}\sum_{m=1}^{\shd} \bell{2s} \:\expect{|\RFcFcoe|^{2}}
   =  \sum_{\ell=0}^{\infty} \bell{2s} \shd \APS{\ell} \le c_{d,s} \sum_{\ell=0}^{\infty} \APS{\ell} \ell^{2s+d-1} < \infty,
\end{equation*}
where the fact that the expectation and series are exchangeable is a consequence of Lebesgue's monotone convergence theorem, and the first inequality uses \eqref{eq:dim.sph.harmon} and \eqref{eq:b_l}.
Thus
 $\sum_{\ell=0}^{\infty}\sum_{m=1}^{\shd} |\bell{s}\RFcFcoe|^{2}<\infty$ $\Pas$.
The Riesz-Fischer theorem then implies
\begin{equation*}
 \sum_{\ell=0}^{\infty}\sum_{m=1}^{\shd} \bell{s} \RFcFcoe \shY(\cdot) \in \Lp{2}{d}, \quad \Pas.
\end{equation*}
That is,
\begin{equation*}
 \RFc \in \sob{2}{s}{d}, \quad \Pas.
\end{equation*}
By Parseval's identity for orthonormal polynomials $\shY$ on $\sph{d}$,
\begin{align*}
 \expect{\norm{\RFc}{\sob{2}{s}{d}}^{2}}
 = \expect{\normB{\sum_{\ell=0}^{\infty}\sum_{m=1}^{\shd} \bell{s} \RFcFcoe \shY}{\Lp{2}{d}}^2}
  &= \expect{\sum_{\ell=0}^{\infty}\sum_{m=1}^{\shd} \bell{2s} |\RFcFcoe|^{2}}\\
  &= \sum_{\ell=0}^{\infty}\sum_{m=1}^{\shd} \bell{2s} \expect{|\RFcFcoe|^{2}}
  = \sum_{\ell=0}^{\infty} \bell{2s} \shd\APS{\ell},
\end{align*}
where the last equality uses \eqref{eq:angular-1}. This completes the proof.
\end{proof}

Under the condition of Theorem~\ref{thm:APS.RFc.in.sob}, since the spherical function $\expect{\RF(\PT{x})}=\expect{\RF(\PT{e}_{d})}$ is a constant function of $\PT{x}\in\sph{d}$ for a $2$-weakly isotropic random field $\RF$, then $\RF \in \sob{2}{s}{d}\; \Pas$, as stated in the following corollary.
\begin{corollary}\label{cor:APS.RF.in.sob}
Let $d\ge2$, $s> d/2$. Let $\RF$ be a $2$-weakly isotropic random field on $\sph{d}$ with the angular power spectrum $\APS{\ell}$ satisfying
 $\sum_{\ell=1}^{\infty}\APS{\ell} \ell^{2s+d-1} < \infty$.
Then $\RF \in \sob{2}{s}{d}\; \Pas$, i.e. $\RF$ has a continuous version $\Pas$.
\end{corollary}

\begin{remark}%\label{rem:angular vs T smoothness}
For a random field on a $\mathcal{C}^{r}$ $d$-manifold, $r>0$, $d\ge1$, Andreev and Lang \cite[Theorem~3.5, p.~6]{AnLa2014} proved a Kolmogorov-Chentsov continuation theorem under a H\"older condition for local regions.
\end{remark}

\subsection{Approximation errors by semidiscrete needlets}
In this section we show how the needlet approximation error for a smooth isotropic random field decays.
\begin{theorem}[Mean $\mathbb{L}_{p}$-error for semidiscrete needlets]\label{thm:err.mean.Lp.neapx}
Let $d\ge2$, $1\le p<\infty$, $s>0$. Let $\RF$ be a $\ceil{p}$-weakly isotropic random field on $\sph{d}$ satisfying $\RF\in \sob{p}{s}{d}$ $\Pas$. Let $\fiN$ be a needlet filter (see \eqref{subeqs:fiN}) satisfying $\fiN\in \CkR$ with $\fis\ge \floor{\frac{d+3}{2}}$.
Then, for $\neord\in \Nz$,
\begin{equation*}
 \expect{\normb{\RF-\neapx(\RF)}{\Lp{p}{d}}^{p}}
 \le c\: 2^{-p\neord s}\: \expect{\norm{\RFc}{\sob{p}{s}{d}}^{p}},
\end{equation*}
where $\RFc$ is given by \eqref{eq:RFc} and the constant $c$ depends only on $d$, $p$, $s$, the filter $\fiN$ and $\fis$.
\end{theorem}
\begin{remark} Theorem~\ref{thm:err.mean.Lp.neapx} with the Fubini theorem and \eqref{eq:E.RF.Sob.APS} and \eqref{eq:sob.norm.sph.LBO} implies that, for $1\le p<\infty$,
\begin{equation*}%\label{eq:RF.err.E.sb2}
 \normb{\RF-\neapx(\RF)}{\Lppsph{p}{d}}
 \le c\: 2^{-\neord s}\:\normb{\sLB{s/2} \:\RFc}{\Lppsph{p}{d}},
\end{equation*}
where the constant $c$ depends only on $d$, $p$, $s$, the filter $\fiN$ and $\fis$.
\end{remark}
\begin{proof} Let $f(\PT{x}):=\expect{\RF(\PT{x})}$, $\PT{x}\in\sph{d}$. By the weak-isotropy of $\RF$, $f(\PT{x})$ is a constant function. Since the needlet approximation reproduces constants exactly, the linearity of $\neapx$ then gives
\begin{equation}\label{eq:neapx.pntw.err.RF.RFc}
    \RF-\neapx(\RF) = \RFc-\neapx(\RFc).
\end{equation}
Since $\RFc\in \sob{p}{s}{d}$ $\Pas$, with Theorem~\ref{thm:needlets.err.W}, this gives
\begin{equation*}%\label{eq:RFc.err.E.sb2}
\expect{\normb{\RF-\neapx(\RF)}{\Lp{p}{d}}^{p}}
 =\expect{\normb{\RFc-\neapx(\RFc)}{\Lp{p}{d}}^{p}}
 \le c\: 2^{-p\neord s}\: \expect{\norm{\RFc}{\sob{p}{s}{d}}^{p}},
 \end{equation*}
where the constant $c$ depends only on $d,p,s,\fiN$ and $\fis$, thus completing the proof.
\end{proof}

The following theorem shows that the semidiscrete needlet approximation error of a $2$-weakly isotropic random field is determined by the rate of decay of the tail of the angular power spectrum.
\begin{theorem}[Mean $\mathbb{L}_{2}$-error for semidiscrete needlets]\label{thm:err.mean.L2.neapx}
Let $d\ge2$, $s>0$. Let $\RF$ be a $2$-weakly isotropic random field on $\sph{d}$ with angular power spectrum $\APS{\ell}$ satisfying $\sum_{\ell=1}^{\infty} \APS{\ell} \ell^{2s+d-1} < \infty$. Let $\fiN$ be a needlet filter (see \eqref{subeqs:fiN}) satisfying $\fiN\in \CkR$ with $\fis\ge \floor{\frac{d+3}{2}}$.
Then, for $\neord\in \Nz$,
\begin{equation*}
 \expect{\normb{\RF-\neapx(\RF)}{\Lp{2}{d}}^{2}}
 \le c\: 2^{-2\neord s}\: \expect{\norm{\RFc}{\sob{2}{s}{d}}^{2}},
\end{equation*}
where $\RFc$ is given by \eqref{eq:RFc} and the constant $c$ depends only on $d$, $s$, the filter $\fiN$ and $\fis$.
\end{theorem}
\begin{proof}
Theorem~\ref{thm:APS.RFc.in.sob} with $\sum_{\ell=0}^{\infty} \APS{\ell} \ell^{2s+d-1} < \infty$ shows that $\RFc(\omega)\in \sob{2}{s}{d}$ for almost every $\omega\in\probSp$, and $\expect{\norm{\RFc}{\sob{2}{s}{d}}^{2}}  < \infty$.
This with Theorem~\ref{thm:needlets.err.W} and \eqref{eq:neapx.pntw.err.RF.RFc} gives
\begin{equation*}%\label{eq:RFc.err.E.sb2}
\expect{\normb{\RF-\neapx(\RF)}{\Lp{2}{d}}^{2}}
=\expect{\normb{\RFc-\neapx(\RFc)}{\Lp{2}{d}}^{2}}
 \le c\: 2^{-2\neord s}\: \expect{\norm{\RFc}{\sob{2}{s}{d}}^{2}},
 \end{equation*}
where the constant $c$ depends only on $d,s,\fiN$ and $\fis$, thus completing the proof.
\end{proof}

Corollary~\ref{cor:APS.RF.in.sob} with Theorem~\ref{thm:needlets.err.W} and \eqref{eq:neapx.pntw.err.RF.RFc} implies the following pointwise approximation error bound for needlet approximations of a $2$-weakly isotropic random field on $\sph{d}$.
\begin{theorem}[Pointwise error for semidiscrete needlets]\label{thm:err.pointwise.L2.neapx} Let $d\ge2$, $s>0$. Let $\RF$ be a $2$-weakly isotropic random field on $\sph{d}$ with angular power spectrum $\APS{\ell}$ satisfying $\sum_{\ell=1}^{\infty} \APS{\ell} \ell^{2s+d-1} < \infty$. Let $\fiN$ be a needlet filter given by \eqref{subeqs:fiN} and satisfying $\fiN\in \CkR$ with $\fis\ge \floor{\frac{d+3}{2}}$.
Then, for $\neord\in\Nz$, $\Pas$,
\begin{equation*}%\label{eq:truncation err pnt-1}
  \normb{\RF-\neapx(\RF)}{\Lp{2}{d}} \le  c\:{2^{-\neord s}}\:\norm{\RFc}{\sob{2}{s}{d}},
\end{equation*}
where $\RFc$ is given by \eqref{eq:RFc} and the constant $c$ depends only on $d$, $s$, $\fiN$ and $\fis$.
\end{theorem}

\subsection{Fully discrete needlet approximation}
For implementation, we follow \cite{WaLeSlWo2016} in using a quadrature rule (which will usually be different from the needlet quadrature) to discretise the needlet coefficients of the semidiscrete needlet approximation in \eqref{eq:neapx.RF}.
Let $\needlet{jk}$ be needlets satisfying \eqref{subeqs:needlets}, and let
\begin{equation}\label{eq:QH}
    \QH := \QH[](N,\ell) := \{(\wH,\pH{i}):i=1,\dots,N\}
\end{equation}
be a \emph{discretisation quadrature rule} that is exact for polynomials of degree up to some $\ell$,
yet to be fixed. Applying the quadrature rule $\QH$ to the needlet
coefficient
  $\InnerL{\RF,\needlet{jk}} = \int_{\sph{d}} \RF(\PT{y})\needlet{jk}(\PT{y}) \IntDiff{y}$,
we obtain the discrete needlet coefficient
\begin{equation*}%\label{eq:f.needlet.discrete.inner.prod}
    \InnerD{\RF,\needlet{jk}} := \sum_{i=1}^{N} \wH\:\RF(\pH{i})\:\needlet{jk}(\pH{i}).
\end{equation*}
This turns the semidiscrete needlet approximation \eqref{eq:neapx.RF} into the \emph{(fully) discrete needlet approximation}
\begin{equation}\label{eq:disneapx.RF}
    \disneapx(\RF) = \sum_{j=0}^{\neord}\sum_{k=1}^{N_{j}}\InnerD{\RF,\needlet{jk}} \needlet{jk}.
\end{equation}

In \cite[Theorem~4.3]{WaLeSlWo2016}, we gave the following convergence result for the discrete needlet approximation of $f\in\sob{p}{s}{d}$.
\begin{theorem}[\cite{WaLeSlWo2016}]\label{thm:dis.needlets.err.Wp} Let $d\geq2$, $1\leq p\leq \infty$, $s>d/p$, $\neord\in\Nz$. Let $\disneapx$ be the fully discrete needlet approximation with needlet filter $\fiN\in \CkR$ and $\fis\ge \floor{\frac{d+3}{2}}$ and with discretisation quadrature $\QH$ exact for degree $3\cdot 2^{\neord-1}-1$. Then, for $f\in \sob{p}{s}{d}$,
    \begin{equation*}
    \normb{f-\disneapx(f)}{\Lp{p}{d}} \leq c\: 2^{-\neord s}\: \norm{f}{\sob{p}{s}{d}},
    \end{equation*}
    where the constant $c$ depends only on $d$, $p$, $s$, $\fiN$ and $\fis$.
\end{theorem}

Theorem~\ref{thm:dis.needlets.err.Wp} implies the following error bound for the discrete needlet approximation of a smooth $\ceil{p}$-weakly isotropic random field, cf. Theorem~\ref{thm:err.mean.Lp.neapx}.
\begin{theorem}[Mean $\mathbb{L}_{p}$-error for discrete needlets]\label{thm:err.mean.Lp.disneapx}
Let $d\ge2$, $1\le p<\infty$, $s>d/p$, $\neord\in \Nz$. Let $\RF$ be a $\ceil{p}$-weakly isotropic random field on $\sph{d}$ satisfying $\RF\in \sob{p}{s}{d}$ $\Pas$. Let $\QH$ be a discretisation quadrature exact for degree $3\cdot 2^{\neord-1}-1$ and let $\fiN$ be a needlet filter given by \eqref{subeqs:fiN} and satisfying $\fiN\in \CkR$ with $\fis\ge \floor{\frac{d+3}{2}}$.
Then
\begin{equation*}
 \expect{\normb{\RF-\disneapx(\RF)}{\Lp{p}{d}}^{p}}
 \le c\: 2^{-p\neord s}\: \expect{\norm{\RFc}{\sob{p}{s}{d}}^{p}},
\end{equation*}
where $\RFc$ is given by \eqref{eq:RFc} and the constant $c$ depends only on $d$, $p$, $s$, the filter $\fiN$ and $\fis$.
\end{theorem}
\begin{remark} Theorem~\ref{thm:err.mean.Lp.disneapx} with the Fubini theorem and \eqref{eq:E.RF.Sob.APS} and \eqref{eq:sob.norm.sph.LBO} implies that
\begin{equation*}%\label{eq:RF.err.E.sb2}
 \normb{\RF-\disneapx(\RF)}{\Lppsph{p}{d}}
 \le c\: 2^{-\neord s}\:\normb{\sLB{s/2} \:\RFc}{\Lppsph{p}{d}},
\end{equation*}
where the constant $c$ depends only on $d$, $p$, $s$, the filter $\fiN$ and $\fis$.
\end{remark}
\begin{proof} Let $f(\PT{x}):=\expect{\RF(\PT{x})}$, $\PT{x}\in\sph{d}$. By the weak-isotropy of $\RF$, $f(\PT{x})$ is a constant function. Since $\disneapx(f)=\neapx(f)$, the discrete needlet approximation reproduces constants exactly. The linearity of $\disneapx$ then gives
\begin{equation}\label{eq:disneapx.pntw.err.RF.RFc}
    \RF-\disneapx(\RF) = \RFc-\disneapx(\RFc).
\end{equation}
Since $\RFc\in \sob{p}{s}{d}$ $\Pas$, Theorem~\ref{thm:dis.needlets.err.Wp} then gives
\begin{equation*}
\expect{\normb{\RF-\disneapx(\RF)}{\Lp{p}{d}}^{p}}
=\expect{\normb{\RFc-\disneapx(\RFc)}{\Lp{p}{d}}^{p}}
 \le c\: 2^{-p\neord s}\: \expect{\norm{\RFc}{\sob{p}{s}{d}}^{p}},
 \end{equation*}
where the constant $c$ depends only on $d,p,s,\fiN$ and $\fis$, thus completing the proof.
\end{proof}

Theorem~\ref{thm:err.mean.Lp.disneapx} implies the following error bound for the discrete needlet approximation of a smooth $2$-weakly isotropic random field, where the condition is stated in terms of angular power spectrum.
\begin{theorem}[Mean $\mathbb{L}_{2}$-error for discrete needlets]\label{thm:err.mean.L2.disneapx} Let $d\ge2$, $s>d/2$, $\neord\in\Nz$. Let $\RF$ be a $2$-weakly isotropic random field on $\sph{d}$ with angular power spectrum $\APS{\ell}$ satisfying $\sum_{\ell=1}^{\infty}\APS{\ell}\ell^{2s+d-1}<\infty$. Let $\QH$ be a discretisation quadrature exact for degree $3\cdot 2^{\neord-1}-1$ and let $\fiN$ be a needlet filter given by \eqref{subeqs:fiN} and satisfying $\fiN\in \CkR$ with $\fis\ge \floor{\frac{d+3}{2}}$. Then,
\begin{equation*}
 \expect{\normb{\RF-\disneapx(\RF)}{\Lp{2}{d}}^{2}}
 \le c\: 2^{-2\neord s}\: \expect{\norm{\RFc}{\sob{2}{s}{d}}^{2}},
\end{equation*}
where $\RFc$ is given by \eqref{eq:RFc} and the constant $c$ depends only on $d$, $s$, filter $\fiN$ and $\fis$.
\end{theorem}
\begin{proof} Theorem~\ref{thm:APS.RFc.in.sob} with $\sum_{\ell=0}^{\infty} \APS{\ell} \ell^{2s+d-1} < \infty$ shows that $\RFc\in \sob{2}{s}{d}$ $\Pas$, and that $\expect{\norm{\RFc}{\sob{2}{s}{d}}^{2}}  < \infty$.
This with Theorem~\ref{thm:err.mean.Lp.disneapx} and \eqref{eq:disneapx.pntw.err.RF.RFc} gives
\begin{equation*}
\expect{\normb{\RF-\disneapx(\RF)}{\Lp{2}{d}}^{2}}
=\expect{\normb{\RFc-\disneapx(\RFc)}{\Lp{2}{d}}^{2}}
 \le c\: 2^{-2\neord s}\: \expect{\norm{\RFc}{\sob{2}{s}{d}}^{2}},
 \end{equation*}
where the constant $c$ depends only on $d,s,\fiN$ and $\fis$, thus completing the proof.
\end{proof}

\begin{remark} When $\RF$ is $2$-weakly isotropic and the angular power spectrum satisfies $\sum_{\ell=1}^{\infty}\APS{\ell} \ell^{2s+d-1}<\infty$ (i.e. $\RF\in \sob{2}{s}{d}$ $\Pas$ by Corollary~\ref{cor:APS.RF.in.sob}), the Sobolev embedding theorem (see Lemma~\ref{lm:embed.sob.sob}) shows that $\Pas$, $\RF\in \sob{p}{\alpha}{d}$, $1\le p<\infty$, $\alpha\le s-d/2+d/p$. The remark following Theorem~\ref{thm:err.mean.Lp.disneapx} then implies
\begin{equation*}
 \normb{\RF-\disneapx(\RF)}{\Lppsph{p}{d}}
 \le c\: 2^{-\neord \alpha}\:\normb{\sLB{\alpha/2} \RFc}{\Lppsph{p}{d}},
\end{equation*}
where the constant $c$ depends only on $d$, $p$, $s$, $\alpha$, $\fiN$ and $\fis$.
\end{remark}

Theorem~\ref{thm:dis.needlets.err.Wp} with Corollary~\ref{cor:APS.RF.in.sob} and \eqref{eq:disneapx.pntw.err.RF.RFc} gives the pointwise approximation error for discrete needlet approximations, cf. Theorem~\ref{thm:err.pointwise.L2.neapx}.
\begin{theorem}[Pointwise error for discrete needlets]\label{thm:err.pointwise.L2.disneapx} Let $d\ge2$, $s>d/2$, $\neord\in\Nz$. Let $\RF$ be a $2$-weakly isotropic random field on $\sph{d}$ with angular power spectrum $\APS{\ell}$ satisfying $\sum_{\ell=1}^{\infty}\APS{\ell}\ell^{2s+d-1}<\infty$. Let $\QH$ be exact for degree $3\cdot 2^{\neord-1}-1$ given $\neord\in\Nz$ and let $\fiN$ be a needlet filter in \eqref{subeqs:fiN} and satisfying $\fiN\in \CkR$ with $\fis\ge \floor{\frac{d+3}{2}}$. Then, $\Pas$,
\begin{equation*}
  \normb{\RF-\disneapx(\RF)}{\Lp{2}{d}}
 \le c\:{2^{-\neord s}}\:\norm{\RFc}{\sob{2}{s}{d}},
\end{equation*}
where $\RFc$ is given by \eqref{eq:RFc} and the constant $c$ depends only on $d$, $s$, $\fiN$ and $\fis$.
\end{theorem}

As shown by \cite{WaLeSlWo2016}, the discrete needlet approximation \eqref{eq:disneapx.RF} is equivalent to filtered hyperinterpolation, which we now introduce.
The \emph{filtered hyperinterpolation approximation} \cite{LeMh2008,SlWo2012} with a filtered kernel $\vdh{L,\fil}$ in \eqref{eq:filter.sph.ker} and discretisation quadrature $\QH$ in \eqref{eq:QH} is
\begin{equation*}%\label{eq:filter.hyper}
    \fihyper{L,\fil,N}(\RF;\PT{x}) := \fihyper[d]{L,\fil,N}(\RF;\PT{x}) := \InnerDb{\RF,\vdh{L,\fil}(\cdot\cdot \PT{x})}
     := \sum_{i=1}^{N} \wH\: \RF(\pH{i})\: \vdh{L,\fil}(\pH{i}\cdot \PT{x}), \quad L\in \Nz.
\end{equation*}
In \cite[Theorem~4.1]{WaLeSlWo2016} it is shown that this is identical to $\disneapx(\RF;\PT{x})$ if $L=2^{\neord-1}$ and $\fil=\fiH$, where $\fiH$ is given by \eqref{eq:fiH}. That is, for a random field $\RF$ on $\sph{d}$ and $\neord\in \Nz$,
\begin{equation}\label{eq:fihyper.disneapx}
    \disneapx(\RF;\omega,\PT{x})=\fihyper{2^{\neord-1},\fiH,N}(\RF;\omega,\PT{x})= \InnerDb{\RF(\omega),\vdh{2^{\neord-1},\fiH}(\cdot\cdot \PT{x})}, \quad \omega\in\probSp,\; \PT{x}\in\sph{d}.
\end{equation}

\subsection{Boundedness of discrete needlet approximation}\label{subsec:Bd.disneapx.Lppsph}
Similarly to the semidiscrete needlet approximation, the $\Lppsph{p}{d}$-norm, $1\le p<\infty$, of $\disneapx(\RF)$ is bounded by the $\Lppsph{\ceil{p}}{d}$-norm of the random field $\RF$ when $\RF$ is $\ceil{p}$-weakly isotropic and the filter $\fiN$ is sufficiently smooth, as proved in the following theorem, cf. Theorem~\ref{thm:Bd.neapx.Lppsph}.
\begin{theorem}[Boundedness of discrete needlet approximation]\label{thm:Bd.disneapx.Lppsph} Let $d\ge 2$, $1\le p< \infty$, $\neord\in\Nz$, $c_{0}>0$. Let $\RF$ be a $\ceil{p}$-weakly isotropic random field on $\sph{d}$. Let $\QH$ be a discretisation quadrature exact for degree at least $c_{0} \:2^{\neord}$ and let $\fiN$ be needlet filter in \eqref{subeqs:fiN} satisfying $\fiN\in\CkR$ and $\fis\ge \floor{\frac{d+3}{2}}$. Then
\begin{equation*}
    \normb{\disneapx(\RF)}{\Lppsph{p}{d}}\le c\:\norm{\RF}{\Lppsph{\ceil{p}}{d}},
\end{equation*}
where the constant $c$ depends only on $d$, $c_{0}$, $\fiN$ and $\fis$.
\end{theorem}
\begin{proof}
For $1\le p<\infty$, by the Fubini theorem and \eqref{eq:fihyper.disneapx},
\begin{align*}%\label{eq:RF.nrm.Lppsph}
  \normb{\disneapx(\RF)}{\Lppsph{p}{d}}^{p}
  &= \int_{\sph{d}}\int_{\probSp}|\disneapx(\RF;\omega,\PT{x})|^{p}\:\Dpb\IntDiff{x}\\
  &= \int_{\sph{d}}\normB{\sum_{i=1}^{N}\wH\:\RF(\pH{i})\vdh{2^{\neord-1},\fiH}(\PT{x}\cdot\pH{i})}{\Lpprob{p}}^{p}\IntDiff{x}\notag\\
  &\le \int_{\sph{d}}\left|\sum_{i=1}^{N}\wH\:\norm{\RF(\pH{i})}{\Lpprob{p}}\bigl|\vdh{2^{\neord-1},\fiH}(\PT{x}\cdot\pH{i})\bigr|\right|^{p}\IntDiff{x},
\end{align*}
where the inequality uses the triangle inequality for $\Lpprob{p}$-norms.
By Theorem~\ref{thm:RF.nrm.Lpprob},
\begin{equation*}
\norm{\RF(\PT{y}_{i})}{\Lpprob{p}}
  \le \norm{\RF(\PT{e}_{d})}{\Lpprob{\ceil{p}}},
\end{equation*}
where $\PT{e}_{d}$ is the north pole of $\sph{d}$. This then gives
\begin{align}\label{eq:disne.RF.Lppsph.nrm}
  \normb{\disneapx(\RF)}{\Lppsph{p}{d}}^{p}
  &\le \norm{\RF(\PT{e}_{d})}{\Lpprob{\ceil{p}}}^{p}\int_{\sph{d}}\left(\sum_{i=1}^{N}\wH\:\bigl|\vdh{2^{\neord-1},\fiH}(\PT{x}\cdot\pH{i})\bigr|\right)^{p}\IntDiff{x}.
\end{align}
Since $\vdh{2^{\neord-1},\fiH}(\PT{x}\cdot\PT{y})$ is a polynomial of $\PT{x}$ or $\PT{y}$ of degree at most $2^{\neord}-1$, the $\mathbb{L}_{1}$ Marciekiewicz-Zygmund inequality \cite[Theorem~2.1]{Dai2006} with \cite[Lemma~1]{Reimer2000},
\cite[Lemma~2]{HeSl2006} and \cite[Lemma~3.2]{BrHe2007}, see also \cite[Theorem~3.3]{Mhaskar2006} and \cite[Theorem~3.1]{MhNaWa2001}, then gives
\begin{equation*}%\label{eq:MZ.ineq.fiker}
\sum_{i=1}^{N}\wH\:\bigl|\vdh{2^{\neord-1},\fiH}(\PT{x}\cdot\pH{i})\bigr|
  \le c_{d,c_{0}}\:\normb{\vdh{2^{\neord-1},\fiH}(\PT{x}\cdot\cdot)}{\Lp{1}{d}}.
\end{equation*}
This together with Theorem~\ref{thm:filter.sph.ker.L1norm.UB}, Lemma~\ref{lm:RF.Lpp.Lppsph.nrms} and \eqref{eq:disne.RF.Lppsph.nrm} gives
\begin{equation*}%\label{eq:RF.nrm.Lppsph}
  \normb{\disneapx(\RF)}{\Lppsph{p}{d}}
  \le c_{d,c_{0},\fiN,\fis}\:\norm{\RF}{\Lppsph{\ceil{p}}{d}},
\end{equation*}
 thus completing the proof.
\end{proof}

\section{Convergence of variance}\label{sec:converge.cov.neapx.RF}
In this section, we prove that variances of errors of needlet approximations $\neapx$ and $\disneapx$ decay at the rate $2^{-\neord s}$ for a $2$-weakly isotropic random field on $\sph{d}$ with angular power spectrum $\APS{\ell}$ satisfying $\sum_{\ell=1}^{\infty}\APS{\ell}\ell^{2s+d-1}<\infty$.

\begin{theorem}[Variance of errors for semidiscrete needlets]\label{thm:var.err.neapx} Let $d\ge2$. Let $\RF$ be a $2$-weakly isotropic random field on $\sph{d}$ with angular power spectrum $\APS{\ell}$ satisfying $\sum_{\ell=1}^{\infty}\APS{\ell}\ell^{2s+d-1}<\infty$. Let $\fiN$ be a needlet filter given by \eqref{subeqs:fiN} and satisfying $\fiN\in \CkR$ with $\fis\ge \floor{\frac{d+3}{2}}$. Then, for $\neord\in\Nz$,
\begin{equation*}
    \var{\normb{\RF - \neapx(\RF)}{\Lp{2}{d}}} \le c\: 2^{-2\neord s}\:\expect{\norm{\RFc}{\sob{2}{s}{d}}^{2}},
\end{equation*}
where $\RFc$ is given by \eqref{eq:RFc} and the constant $c$ depends only on $d,s,\fiN$ and $\fis$.
\end{theorem}
\begin{proof} Jensen's inequality \eqref{eq:Jensen.ineq.probab} with Theorem~\ref{thm:err.mean.L2.neapx} gives
\begin{equation*}
    \expect{\normb{\RF-\neapx(\RF)}{\Lp{2}{d}}}
    \le \left\{\expect{\normb{\RF-\neapx(\RF)}{\Lp{2}{d}}^{2}}\right\}^{1/2}
    \le c_{d,s,\fiN,\fis}\: 2^{-\neord s}\:\left(\expect{\norm{\RFc}{\sob{2}{s}{d}}^{2}}\right)^{1/2}.
\end{equation*}
This with Theorem~\ref{thm:err.mean.L2.neapx} (again) and the property of the variance gives
\begin{align*}
    \var{\normb{\RF - \neapx(\RF)}{\Lp{2}{d}}}
        &= \expect{\normb{\RF - \neapx(\RF)}{\Lp{2}{d}}^{2}} - \left\{\expect{\normb{\RF - \neapx(\RF)}{\Lp{2}{d}}}\right\}^{2}\\[1mm]
        &\le c\: 2^{-2\neord s}\:\expect{\norm{\RFc}{\sob{2}{s}{d}}^{2}},
\end{align*}
where the constant $c$ depends only on $d,s,\fiN$ and $\fis$.
\end{proof}

A similar result to Theorem~\ref{thm:var.err.neapx} holds for fully discrete needlets, as follows.
\begin{theorem}[Variance of errors for fully discrete needlets]\label{thm:var.err.disneapx} Let $d\ge2$, $s>d/2$, $\neord\in\Nz$. Let $\RF$ be a $2$-weakly isotropic random field on $\sph{d}$ with angular power spectrum $\APS{\ell}$ satisfying $\sum_{\ell=1}^{\infty}\APS{\ell}\ell^{2s+d-1}< \infty$. Let $\QH$ be a discretisation quadrature exact for degree $3\cdot 2^{\neord-1}-1$ and let $\fiN$ be a needlet filter given by \eqref{subeqs:fiN} and satisfying $\fiN\in \CkR$ with $\fis\ge \floor{\frac{d+3}{2}}$. Then
\begin{equation*}
    \var{\normb{\RF - \disneapx(\RF)}{\Lp{2}{d}}} \le c\: 2^{-2\neord s}\:\expect{\norm{\RFc}{\sob{2}{s}{d}}^{2}},
\end{equation*}
where $\RFc$ is given by \eqref{eq:RFc} and the constant $c$ depends only on $d,s,\fiN$ and $\fis$.
\end{theorem}
The proof, using Theorem~\ref{thm:err.mean.L2.disneapx} and Jensen's inequality, is similar to that for Theorem~\ref{thm:var.err.neapx}.

\section{Numerical examples}\label{sec:numer.examples}
In this section, we present an algorithm for the implementation of discrete needlet approximations of random fields on $\sph{d}$. Examples for isotropic Gaussian random fields on $\sph{2}$ are given, supporting the results in previous sections.

\subsection{Calculating a needlet approximation}
Consider computing the discrete needlet approximation $\disneapx(\RF;\widetilde{\PT{x}}_{i})$ of order $\neord\in\Nz$ with needlet filter $\fiN$ at a set of points $\{\widetilde{\PT{x}}_{i}: i=1,\dots,\widetilde{N}\}$. The needlet quadrature rules $\{(\wN,\pN{jk}):k=1,\dots,N_{j}\}$ are,
by definition, exact for polynomials of degree $2^{j+1}-1$.
\begin{algorithm}\label{alg:disneapx.RF}
Choose $\sampnum$ independent samples $\sampdis_{1},\dots,\sampdis_{\sampnum}$
and generate the Fourier coefficients $a_{\ell m}$ as in Section~\ref{S:GenGRF} below.
%as above, see \eqref{eq:var.APS.GRF.tr.KL} and \eqref{eq:RF.tr.KL}.

\begin{enumerate}
\item For every $\sampdis_{n}$, $n=1,\dots,\sampnum$, compute the analysis and synthesis for $\RF(\sampdis_{n})$.
\item Analysis: Compute the discrete needlet coefficients $\InnerD{\RF(\sampdis_{n}), \needlet{jk}}$, $k=1,\dots,N_{j},\;j=0,\dots,\neord$, using a discretization quadrature rule $\QH:=\QH[](N,3\cdot 2^{\neord-1}-1):=\{(\wH,\pH{i}):i=1,\dots,N\}$.
\item Synthesis: Compute the discrete needlet approximation $\sum_{j=0}^{\neord}\sum_{k=1}^{N_{j}}\InnerD{\RF(\sampdis_{n}), \needlet{jk}} \needlet{jk}(\widetilde{\PT{x}}_{i})$, $i=1,\dots, \widetilde{N}$.
\end{enumerate}
\end{algorithm}
Since Algorithm~\ref{alg:disneapx.RF} is actually the repeated use of \cite[Algorithm~5.1]{WaLeSlWo2016}, for the detailed description of quadrature rules and filters we refer to \cite[Section~5]{WaLeSlWo2016}.

In the numerical experiments, symmetric spherical designs \cite{Womersley_ssd_URL} are used
for both needlet quadrature rules and the discretisation quadrature $\QH$, and the sizes
of the needlet quadratures for all levels $j$, $j=0,\ldots,7$,  are given in Table~\ref{tab:Nj}.
A needlet filter $h \in \CkR[5]$ using piecewise polynomials is used as in \cite{WaLeSlWo2016}.

\begin{table}
\begin{center}
\begin{tabular}{lllllllll}
\hline
Level $j$ & 0 & 1 & 2  &  3  & 4   & 5    & 6    & 7 \\
\hline
$N_j$     & 2 & 6 & 32 & 120 & 498 & 2018 & 8130 & 32642 \\
\hline
\end{tabular}
\caption{The number of needlets at different levels}\label{tab:Nj}
\end{center}
\end{table}

\subsection{Generating Gaussian random fields}\label{S:GenGRF}
We use Gaussian random fields on $\sph{2}$ as numerical examples.
Let $a_{\ell,m}$, $\ell\in\Nz$, $m=1,\dots,2\ell+1$, be independent random variables on $\Rone$ following a normal distribution $\normal{\mu_{\ell}}{\sigma_{\ell}^{2}}$ with mean $\mu_{\ell}:=0$ for $\ell\ge1$ and variance
\begin{equation}\label{eq:var.APS.GRF.tr.KL}
    \sigma_{\ell}^{2}:= \APS{\ell},\;  \ell \ge0,
    \quad \hbox{where~}
    \APS{\ell}\hspace{-0.8mm}:= \APS{\ell,\delta}\hspace{-0.8mm}:= 1/(1+\delta \ell)^{2s+2}.
\end{equation}
The scaling factor $\delta>0$ is used to delay the decay of the angular power spectrum $\APS{\ell}$ in \eqref{eq:var.APS.GRF.tr.KL},
increasing the importance of high frequency components as $\delta$ decreases.
Small values of $\delta$ correspond to a short correlation length (in the sense of geodesic distance)
making the random field ``rougher'' and more difficult to approximate,
but $\delta$ does not change the Sobolev index of the random field.

Let
\begin{equation}\label{eq:RF.tr.KL}
 \RF_{\delta,\KLtr}(\omega,\PT{x}) := \sum_{\ell=0}^{\KLtr}\sum_{m=1}^{2\ell+1} a_{\ell m}(\omega) Y_{\ell,m}(\PT{x}),
\end{equation}
so that in the notation of Section 4 we have
$a_{\ell m} = (\widehat{T}_{\delta,\KLtr})_{\ell m}$.
By \cite[Theorem~6.12, Remark~6.13]{MaPe2011}, given $\delta>0$ and $\KLtr\in \Nz$, the $\RF_{\delta,\KLtr}$ in \eqref{eq:RF.tr.KL} is a Gaussian random field on $\sph{2}$ with expectation
\begin{equation*}%\label{eq:E.RF.tr.KL}
    \expect{\RF_{\delta,\KLtr}(\PT{x})}=\expect{a_{0 1}Y_{0,1}(\PT{x})} = \expect{a_{0 1}} = \mu_{0},\quad \PT{x}\in\sph{2},
\end{equation*}
where we used the linearity of expectation and $\expect{a_{\ell m}}=0$ for $\ell\ge1$.

\subsection{Errors of discrete needlet approximation}
\begin{figure}
  \centering
  \begin{minipage}{\textwidth}
  \begin{minipage}{0.485\textwidth}
  \centering
  %trim option's parameter order: left bottom right top
  \includegraphics[trim = 0mm 0mm 0mm 0mm, width=1.03\textwidth]{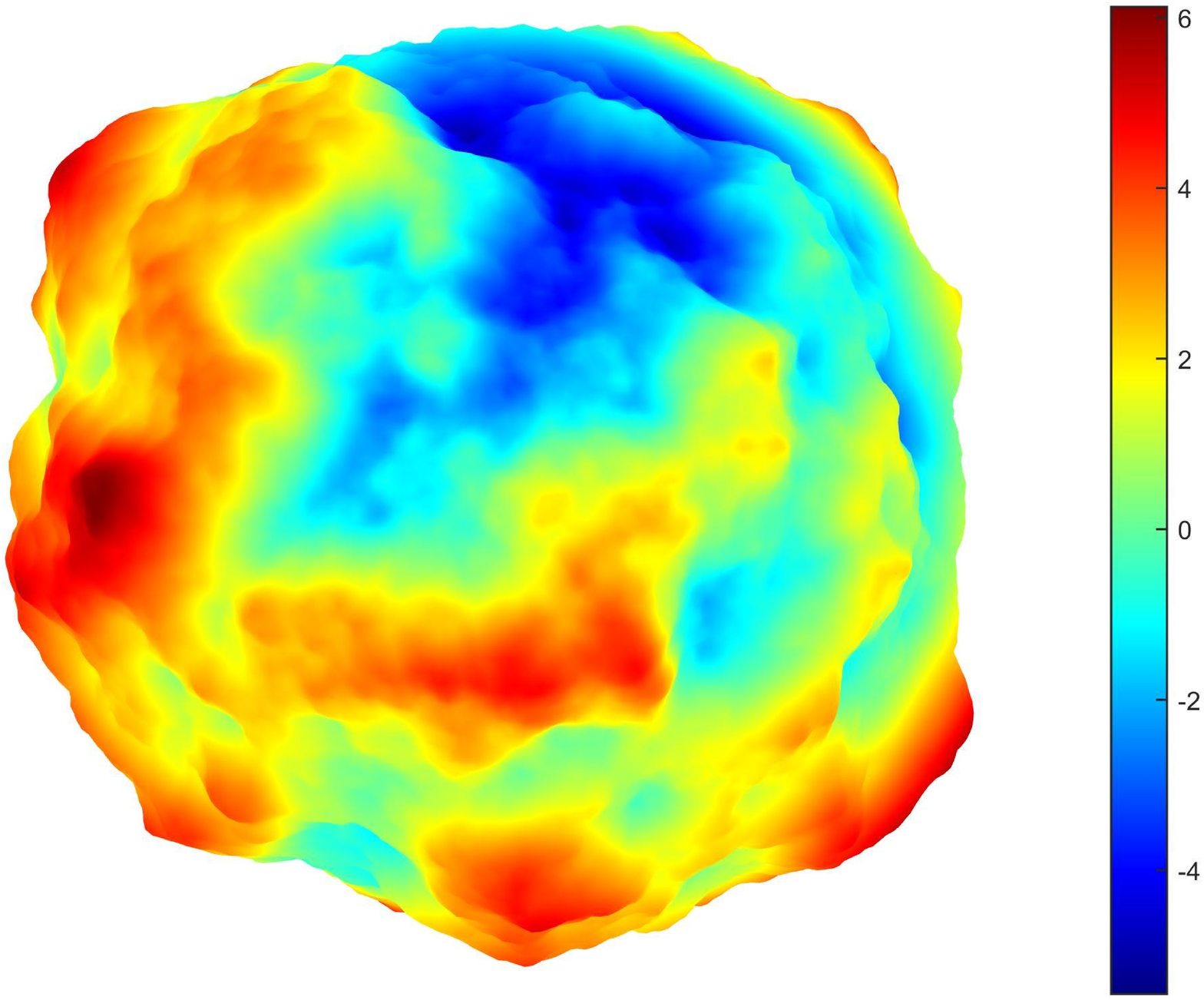}\\[-5mm]
  \subcaption{Original~~~~~}\label{fig:scGRF.one.reali}
  \end{minipage}
  \hspace{0.02\textwidth}
  \begin{minipage}{0.485\textwidth}
  \centering
  %trim option's parameter order: left bottom right top
  \includegraphics[width=\textwidth]{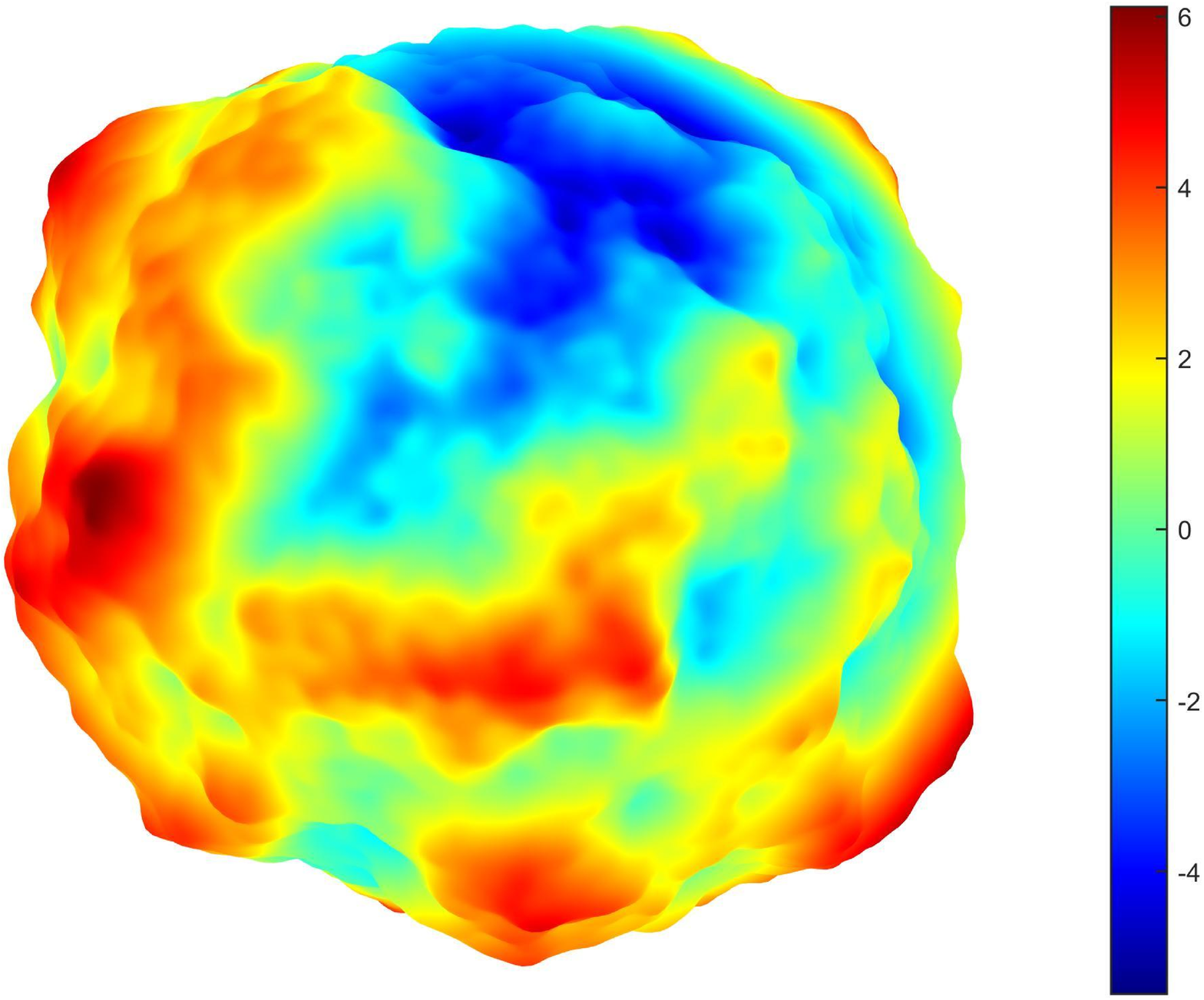}\\[-4mm]
  \subcaption{Fully discrete approximation}\label{fig:disne.scGRF.one.reali}
  \end{minipage}
  \end{minipage}\\[2mm]
  \begin{minipage}{0.5\textwidth}
  \centering
  %trim option's parameter order: left bottom right top
  \includegraphics[trim = 0mm 0mm 0mm 0mm, width=1.05\textwidth]{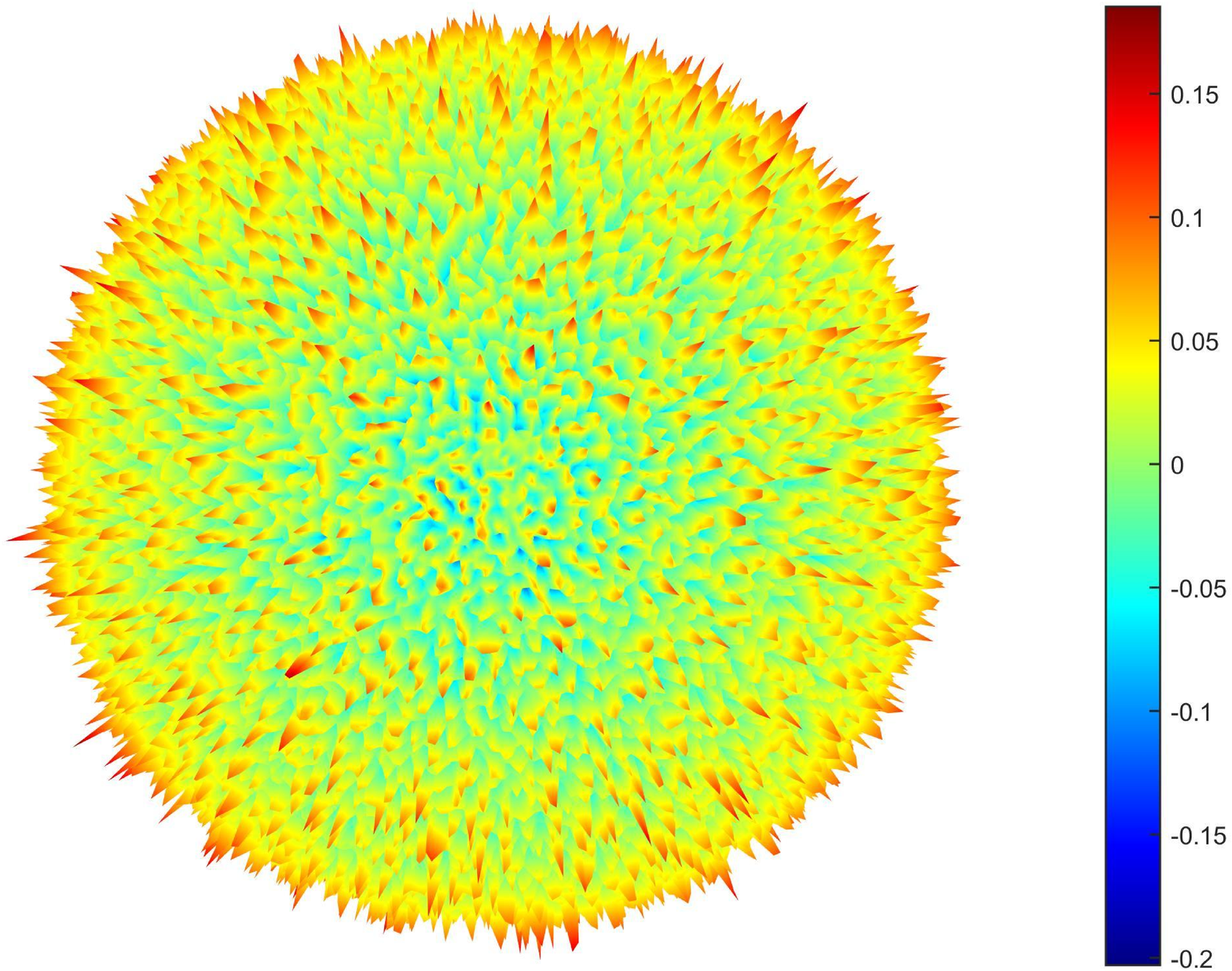}\\[-4mm]
  \subcaption{Pointwise approximation errors}\label{fig:err.disne.scGRF.J6}
  \end{minipage}\\[-1mm]
  \caption{(a) One realisation of the scaled centered Gaussian random field $\RF_{\delta,\KLtr}$,
$\delta=1/5$, $s=1.5$, $\KLtr=300$ so that $\RF_{\delta,\KLtr}\in \sob{2}{s}{2}$, $\Pas$.
(b) Fully discrete needlet approximation with $\neord = 7$.
(c) Needlet approximation errors $\RF(\PT{x}) - \disneapx(\RF;\PT{x})$.} \label{fig:disne.scGRF.J6}
\end{figure}
\begin{figure}
  \centering
  \begin{minipage}{0.48\textwidth}
  \centering
  %trim option's parameter order: left bottom right top
  \includegraphics[trim = 0mm 0mm 0mm 3mm, width=0.935\textwidth]{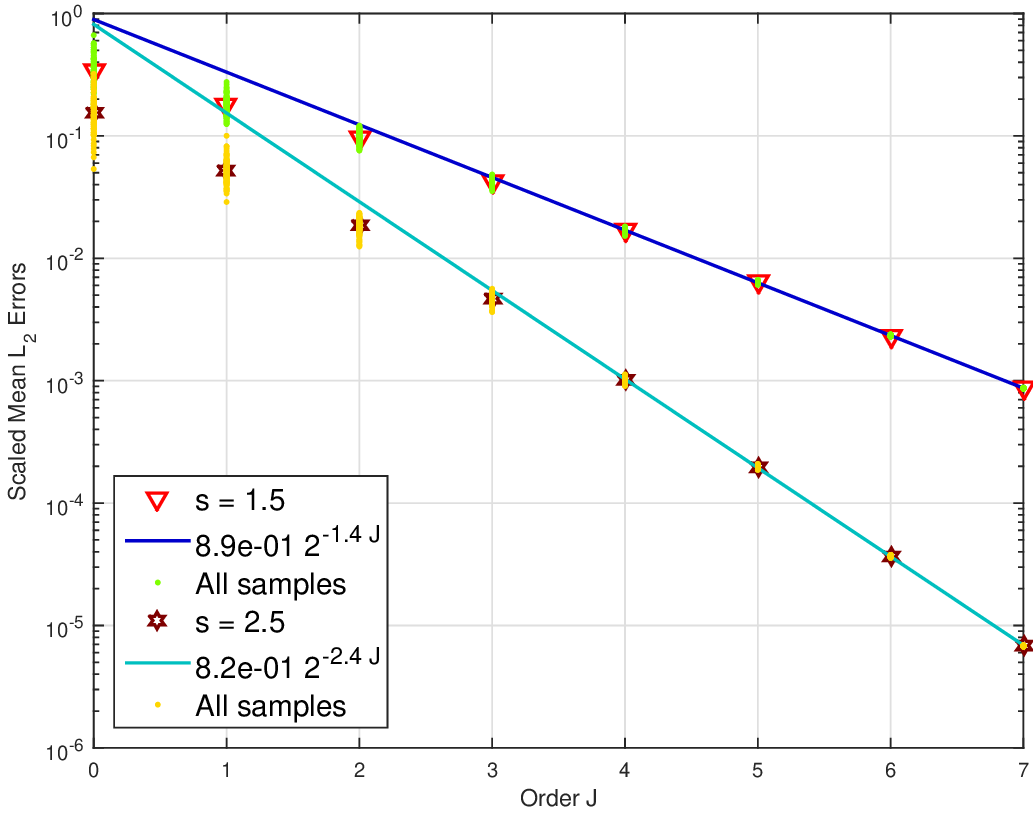}\\[1mm]
  %%RSW1%%\includegraphics[trim = 0mm 0mm 0mm 0mm, width=0.93\textwidth]{dN_ScGRF100_mu0_delta1_J7_QHSD_N18338_QNmSD_L301.eps}\\[1mm]
  \subcaption{$\delta=1$}\label{fig:L2err.disne.scGRF.J7.del1}
  \end{minipage}
  \begin{minipage}{0.48\textwidth}
  \centering
  %trim option's parameter order: left bottom right top
  \includegraphics[trim = 0mm 0mm 0mm 0mm, width=0.9\textwidth]{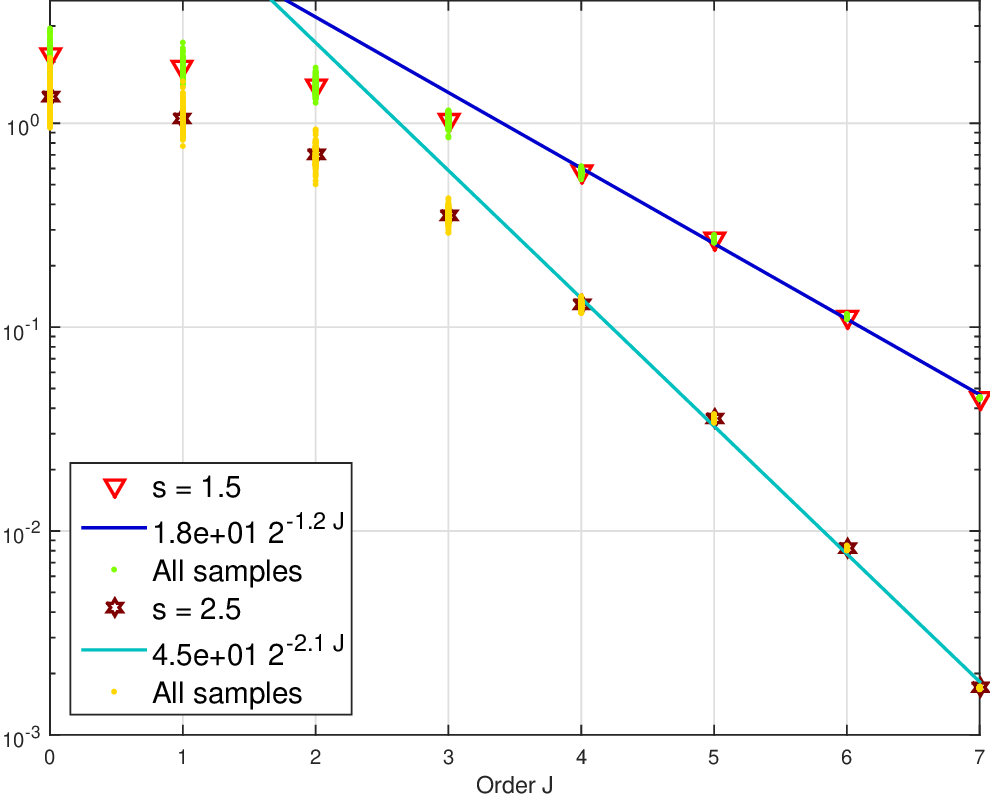}\\[1mm]
  %%RSW1%%\includegraphics[trim = 0mm 0mm 0mm 0mm, width=0.9\textwidth]{dN_ScGRF100_mu0_delta5_J7_QHSD_N18338_QNmSD_L301_ylarm.eps}\\[1mm]
  \subcaption{$\delta=1/5$}\label{fig:L2err.disne.scGRF.J7.del5}
  \end{minipage}
  \caption{Mean $\mathbb{L}_{2}$-errors for discrete needlet approximations of $\RF_{\delta,\KLtr}$, $\KLtr=300$, $\neord=0,\dots,7$, $\fiN\in \CkR[5]$, $s=1.5,2.5$}\label{fig:L2err.disne.scGRF.J7}
\end{figure}
We consider centered Gaussian random fields of \eqref{eq:RF.tr.KL}, i.e. $\mu_{0}=0$ and $a_{\ell m}$
has variance $A_\ell$, with $A_\ell$ as in \eqref{eq:var.APS.GRF.tr.KL}.
Figure~\ref{fig:scGRF.one.reali} shows
one realisation $\RF_{\delta,\KLtr}(\sampdis)$ with $\delta=1/5$, $s=1.5$ and $\KLtr=300$,
while Figure~\ref{fig:disne.scGRF.one.reali}
shows the fully discrete needlet approximation $\disneapx(\RF_{\delta,\KLtr};\sampdis)$
with $J=7$ (which is a polynomial of degree $127$).
Figure~\ref{fig:err.disne.scGRF.J6} shows the pointwise error $\disneapx(\RF_{\delta,\KLtr};\sampdis)-\RF_{\delta,\KLtr}(\sampdis)$ as a function on $\sph{2}$. Figure~\ref{fig:disne.scGRF.J6} demonstrates that using a discrete needlet approximation up to order $7$ can faithfully reconstruct a realisation of a Gaussian random field with low smoothness, short correlation length and high frequency components, losing only some fine details. The approximation errors in Figure~\ref{fig:err.disne.scGRF.J6} come from the high frequency components in the Gaussian random field and are as good as can be expected, with positive and negative oscillations of roughly equal magnitudes. These missing details can be retrieved by adding higher levels in the needlet approximation.

Computationally we need to both estimate the expected value
and approximate the integral over the sphere for the $\Lp{2}{2}$ norm.
The first is done using the sample mean of $\Lp{2}{2}$-errors of $\disneapx(\RF;\sampdis_{n})$ for $n=1,\dots,\sampnum$.
The second is done using another quadrature rule $\bigl\{(\weval,\peval{i}):i=1,\dots,\Neval\bigr\}$
as follows.
%for the $\Lp{2}{2}$-error of $\disneapx(\RF;\sampdis_{n})$ for each $\sampdis_{n}$ as follows.
%%RSW1%%For a random field $\RF$, the $\Lppsph{2}{2}$-errors of $\disneapx(\RF)$ can be approximated as
%%RSW1%%and the $\Lp{2}{2}$-error of $\disneapx(\RF;\sampdis_{n})$ for each $\sampdis_{n}$ is
\begin{align}\label{eq:disne.RF.numer.approx}
    \expect{\normb{\RF - \disneapx(\RF)}{\Lp{2}{2}}^{2}}
    &= \expect{\int_{\sph{2}} \bigl|\RF(\PT{x}) - \disneapx(\RF;\PT{x}) \bigr|^{2} \IntDiff[2]{x}}\notag\\
    &\approx \expect{\sum_{i=1}^{\Neval} \weval\: \bigl(\RF(\peval{i}) - \disneapx(\RF;\peval{i})\bigr)^{2}}\notag\\
    &\approx \frac{1}{\sampnum}\sum_{n=1}^{\sampnum}\sum_{i=1}^{\Neval} \weval\:
     \bigl(\RF(\sampdis_{n},\peval{i}) - \disneapx(\RF;\sampdis_{n},\peval{i}) \bigr)^{2}.
\end{align}
To illustrate the convergence order in Theorem~\ref{thm:err.mean.L2.disneapx} we use the root mean square error
%%RSW1%%scaled mean $\mathbb{L}_{2}$-error,
\begin{equation}\label{eq:relative.err.RF}
   %%RSW%%\mbox{errsc} := \sqrt{\frac{\expect{\normb{\RF-\disneapx(\RF)}{\Lp{2}{2}}^{2}}}{\expect{\norm{\RF}{\sob{2}{s}{2}}^{2}}}}.
   \err := \sqrt{\expect{\normb{\RF-\disneapx(\RF)}{\Lp{2}{2}}^{2}}}.
\end{equation}
%%RSW1%%To estimate the denominator on the right-hand side of \eqref{eq:relative.err.RF},
%%RSW1%%note that when $\ell>\KLtr$, $m=1,\dots,2\ell+1$, by \eqref{eq:RF.tr.KL} we have $(\textcolor{red}{\widehat{T}_{\delta,\KLtr})}=0$.
%%RSW1%%This with Lemma~\ref{lm:orth.Fcoe.RF} gives $\APS{\ell}=0$, $\ell>\KLtr$, and with \eqref{eq:E.RF.Sob.APS} then gives the expectation of the Sobolev norm of the centered Gaussian random field $\RF_{\delta,\KLtr}$:
%%RSW1%%\begin{equation*}
%%RSW1%%    \expect{\norm{\RF_{\delta,\KLtr}}{\sob{2}{s}{2}}^{2}} =
%%RSW1%%    \sum_{\ell=0}^{M} \bell{2s} (2\ell+1) \APS{\ell}
%%RSW1%%    \textcolor{red}{ = \sum_{\ell=0}^{M} \frac{(1+\ell(\ell+1))^{s}(2\ell+1)}{(1+\delta\ell)^{2s+2}}}.
%%RSW1%%\end{equation*}

Figure~\ref{fig:L2err.disne.scGRF.J7} shows mean $\mathbb{L}_{2}$-errors \eqref{eq:relative.err.RF}
of the discrete needlet approximation $\disneapx$ with order $\neord=0,\dots,7$ for
$\RF_{\delta,\KLtr}$ with $\mu_{0}=0$, $\delta=1,1/5$, $\KLtr=300$ and $s=1.5$ and $2.5$,
using the formula \eqref{eq:disne.RF.numer.approx} with $\sampnum=100$
and using a symmetric $301$-design~\cite{Womersley_ssd_URL} with $\Neval = 45,454$ to estimate the $L_2$ error.
The graphs show convincingly that mean $\mathbb{L}_{2}$-errors decay at a rate close to $2^{-\neord s}$
for both $s=1.5$ and $2.5$, supporting the assertion of Theorem~\ref{thm:err.mean.L2.disneapx}.
When the scaling factor $\delta$ is smaller, i.e. when the random field is ``rougher'',
higher levels are required before the asymptotic rate becomes apparent.
The graphs also show that in both cases the sample variances converge to zero as $\neord$ increases, which supports Theorem~\ref{thm:var.err.disneapx}.

\subsection{Comparison of needlet and Fourier approximations}

In this section we compare the needlet and truncated Fourier approximations.
In practice we compare the discrete versions of the two approximation schemes:
the discrete needlet approximation on the one hand,
and on the other the so-called ``hyperinterpolation'' approximation \cite{Sloan1995,SlWo2000,HeSl2007}.
The \emph{hyperinterpolation} approximation of degree $L$, denoted by $\hyper(\RF;\omega,\PT{x})$,
truncates the Fourier expansion in \eqref{eq:KL.expan.RF} at $\ell=L$
and replaces the integral in the inner product $\InnerL{\RF,\shY}$ by
a quadrature rule $\QH:=\{(\wH,\pH{i})|i=1,\dots,N\}$
on $\sph{2}$ exact for degree $2L$.

In making this comparison it is not immediately clear what is a fair experiment.
The level $J$ needlet approximation is a polynomial of degree $2^J-1$,
so for $J = 7$ (as in all the experiments in this section)
the needlet is a polynomial of degree $127$.
On the other hand, by design it reproduces spherical harmonics of degree up to degree only $2^{J-1}$, or $64$ in this case.
In the following experiments we choose to take the degree of the hyperinterpolation approximation to be $128$.
In this way a clear advantage is given to the Fourier projection (as represented by its surrogate, hyperinterpolation),
since it is well known that the Fourier projection is the optimal $L_2$ approximation.

On the other hand, needlets have a clear advantage when the aim is not a global approximation of uniform quality,
but rather a high quality approximation over a limited region of the sphere.

In the following experiment the function to be approximated is defined by
\begin{equation}\label{eq:RF.KL.plus.fhump}
  \RF_{\delta,\KLtr}^{(1)}(\omega,\PT{x}) := \RF_{\delta,\KLtr}(\omega,\PT{x}) + f_{\mbox{ccap}}(\PT{x}) ,
  \quad \omega\in\probSp,\; \PT{x}\in\sph{2},
\end{equation}
where, using $\dist{\PT{x},\PT{x}_{0}}:=\arccos(\PT{x}\cdot\PT{x}_{0})$,
\begin{equation*}
  f_{\mbox{ccap}}(\PT{x}) := \left\{
  \begin{array}{ll}
  %\cos(4\:\dist{\PT{x},\PT{x}_{0}}), & \dist{\PT{x},\PT{x}_{0}}\le \pi/8,\\
  \cos\left(\frac{\pi}{2} \frac{\dist{\PT{x},\PT{x}_{0}}}{r}\right) & \dist{\PT{x},\PT{x}_{0}}\le r, \\
  0, & \dist{\PT{x},\PT{x}_{0}} > r;
  \end{array}
  \right.
\end{equation*}
is the $C^0$ cosine cap function supported on a spherical cap of radius $r$ about $\PT{x}_0$ (see \cite{AnChSlWo2012}).
The function~(\ref{eq:RF.KL.plus.fhump}) can be interpreted as a random field in which the mean field is not constant over the sphere,
but instead has a cosine cap in a neighbourhood of $\PT{x}_0$, which we take to be the north pole.
%This is an isotropic Gaussian random field with mean $\mu_{0} + f_{\mbox{ccap}}(\PT{x})$ and variance $\APS{\ell}$.
The function $f_{\mbox{ccap}}$ is in $\ContiSph[2]{}$ but not in any Sobolev space $\sob{2}{s}{2}$ for $s>1$.
The boundary of the support where $\dist{\PT{x},\PT{x}_{0}} = r$ is the most difficult part to approximate.
Figure~\ref{fig:GRFhump_original} shows one realisation of this random field
$\RF_{\delta,M}^{(1)}$ in \eqref{eq:RF.KL.plus.fhump} with $\delta=1$, $M=130$ and $s=1.5$.

We take the point of view that there is a particular interest in approximating the random field in the polar region,
where the cosine cap occurs.
For that reason we take a relatively low order needlet approximation globally, with $J = 4$,
but in the region where $\arccos(\PT{x}\cdot\PT{x}_{0}) \leq \pi/3$ we increase the order to $J = 7$.
%%RSW%%The resulting approximation is shown in Figure~\ref{fig:GRFhump_locNeed7_tr4}, and
%%RSW%%the error in this approximation is shown in Figure~\ref{fig:GRFhump_locNeed7_error}.
%%RSW%%
The localised needlet approximation in Figure~\ref{fig:GRFhump_locNeed7_tr4} is excellent in the polar cap
where needlets of all levels up to $J = 7$ are used.
It still achieves a good approximation over the remainder of the sphere where
only needlets up to level $4$ are used, as illustrated in the error plot in Figure~\ref{fig:GRFhump_locNeed7_error}.
%%RSW1%%This localised needlet strategy has good approximation to the realisation of the random field in
%%RSW1%%Figure~\ref{fig:GRFhump_original} over the entire sphere using levels up to $J = 4$.
%%RSW1%%Moreover, the approximation is terrific in the cap where needlets of all levels up to $J = 7$ are used.
We refer to the Riemann localisation of the filtered approximation \cite{WaSlWo2016} for an explanation
of the behaviour of local approximation by needlets.
With symmetric spherical designs, the localised (in a cap of radius $\pi/3$)
needlet approximation in Figure~\ref{fig:GRFhump_locNeed7_tr4} uses a total of $11,341$ needlets.
%%RSW1%%truncating at levels $5$ to $7$ with truncation radius $\pi/3$ uses $11341$ needlets.
This is around $26\%$ of the $43,448$ needlets used in fully discrete needlet approximation
for all levels up to $J = 7$ over the whole sphere, see Table~\ref{tab:Nj}.
%%RSW1%%Thus, the localised needlet approximation has much lower computational cost than the fully discrete needlet approximation
%%RSW1%%since the levels $5$ to $7$ use only a fraction of the full set of needlets,
%%RSW1%%approximately $|\scap{\PT{x}_{0},r}|/|\sph{2}|=(1-\cos(r))/2$ (about $25\%$ when $r=\pi/3$),
%%RSW1%%where $|\scap{\PT{x}_{0},r}|$ is the area of the cap $\scap{\PT{x}_{0},r}$ with center $\PT{x}_{0}$ and radius $r$.
Thus needlets efficiently allow the local concentration of points in regions of most interest.

Finally, Figure~\ref{fig:GRFhump_Hyper7_error} shows the error in the hyperinterpolation approximation of degree $128$.
Of course in the Fourier case there is no capacity for localising the approximation.

\begin{figure}
  \centering
   \begin{minipage}{\textwidth}
  \begin{minipage}{0.485\textwidth}
  \centering
  %trim option's parameter order: left bottom right top
  \includegraphics[trim = 0mm 0mm 0mm 0mm, width=0.9\textwidth]{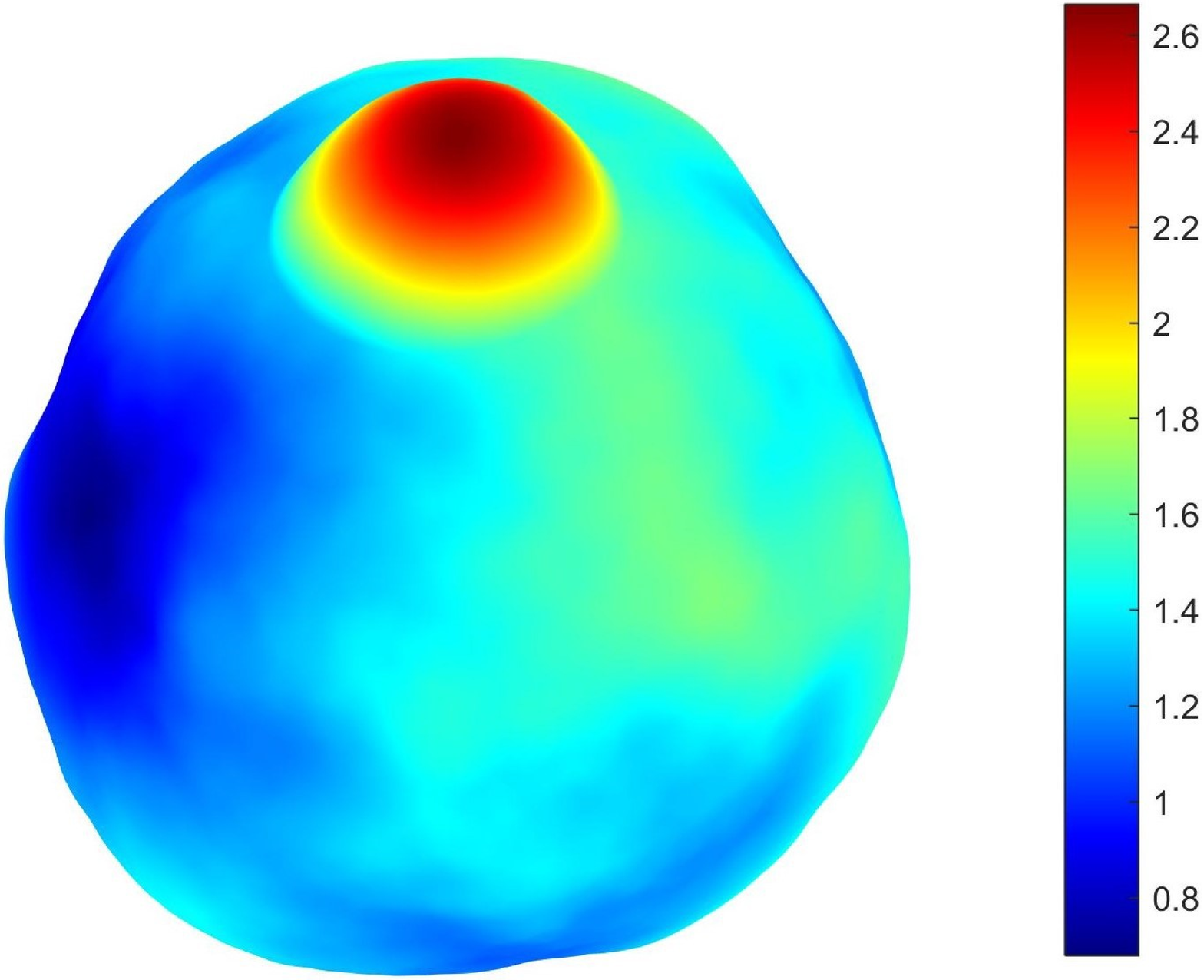}\\[1mm]
  %\subcaption{A realisation of Gaussian random field $\RF_{\delta,M}^{(1)}$}\label{fig:GRFhump_original}
  \subcaption{A realisation of Gaussian random field $\RF_{\delta,M}^{(1)}$}\label{fig:GRFhump_original}
  \end{minipage}
  \hspace{0.02\textwidth}
  \begin{minipage}{0.485\textwidth}
  \centering
  %trim option's parameter order: left bottom right top
  \includegraphics[trim = 0mm 0mm 0mm 0mm, width=0.9\textwidth]{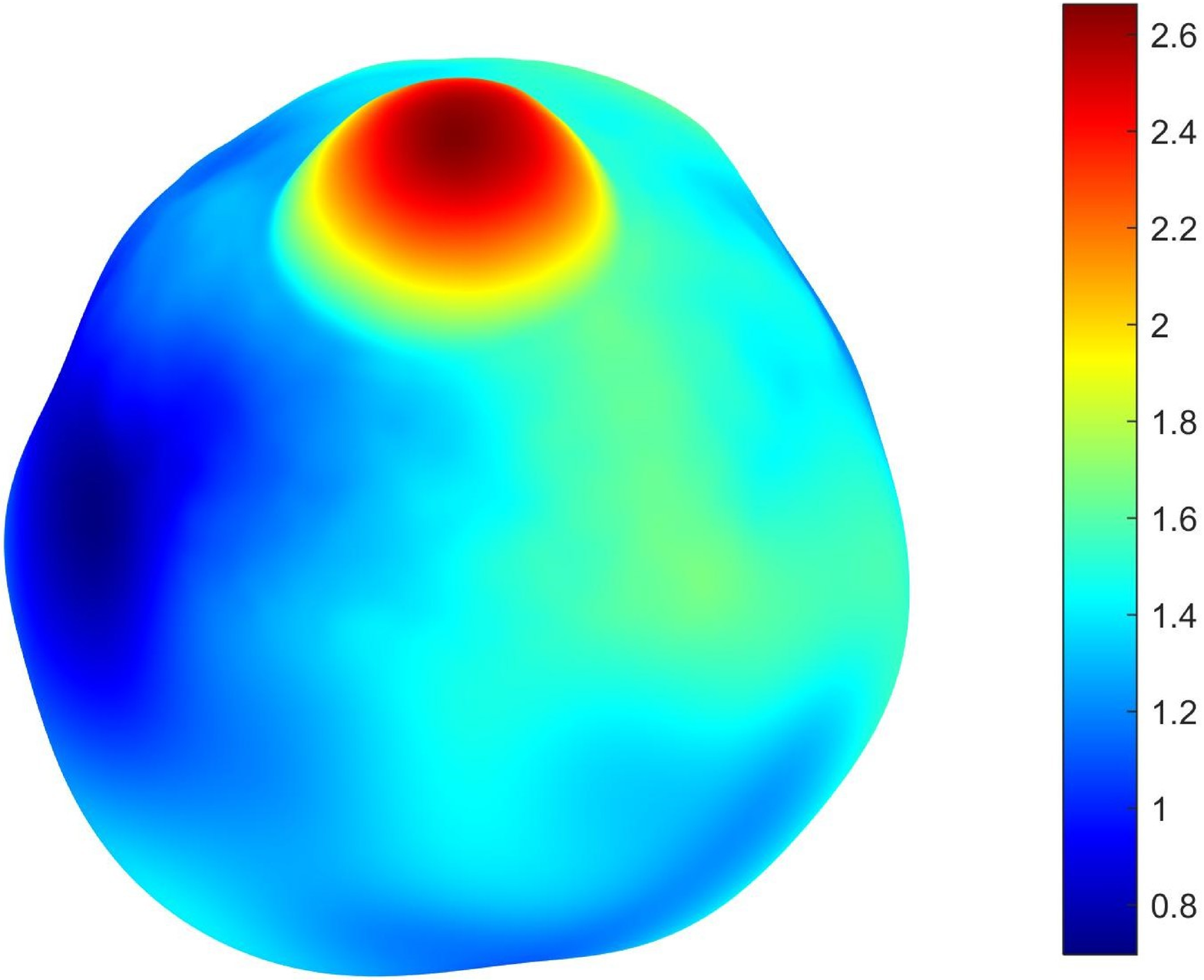}\\[1mm]
  \subcaption{Localised needlet approximation}\label{fig:GRFhump_locNeed7_tr4}
  \end{minipage}
  \end{minipage}\\[4mm]
  \begin{minipage}{\textwidth}
  \begin{minipage}{0.485\textwidth}
  \centering
  %trim option's parameter order: left bottom right top
  \includegraphics[trim = 0mm 0mm 0mm 0mm, width=0.9\textwidth]{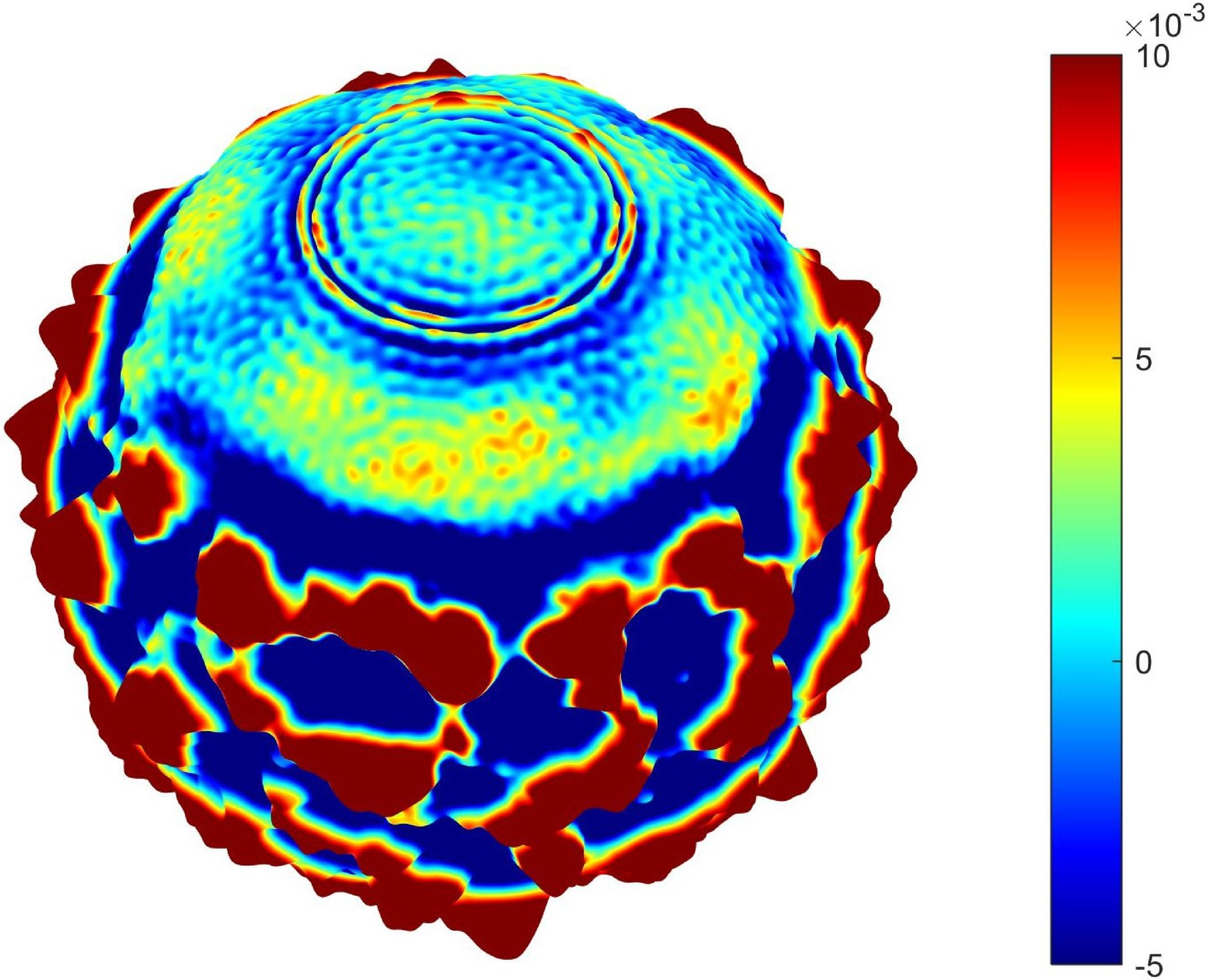}\\[1mm]
  \subcaption{Local needlet approximation error}\label{fig:GRFhump_locNeed7_error}
  \end{minipage}
  \hspace{0.02\textwidth}
  \begin{minipage}{0.485\textwidth}
  \centering
  %trim option's parameter order: left bottom right top
  \includegraphics[trim = 0mm 0mm 0mm 0mm, width=0.9\textwidth]{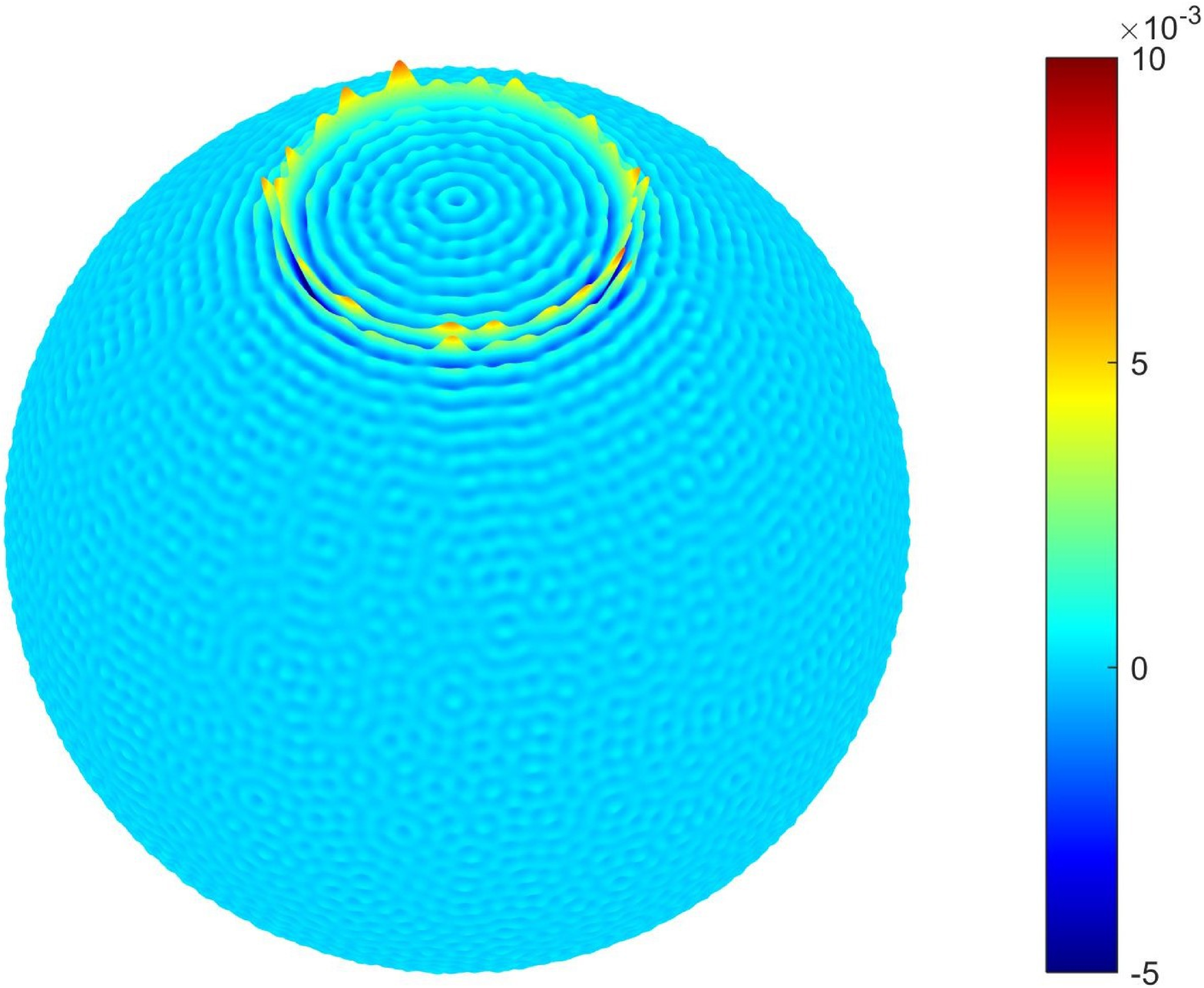}\\
  \subcaption{Hyperinterpolation approximation error}\label{fig:GRFhump_Hyper7_error}
  \end{minipage}
  \end{minipage}\\
  \caption{(a) Realisation of the random field $\RF_{\delta,M}^{(1)}$, $\delta=1$, $s=1.5$, $M=130$.
(b) Localised needlet approximation of (a) using all needlets for levels $0$ to $4$ and
using needlets whose centers are in the cap of radius $\pi/3$ for levels $5$ to $7$, with needlet filter $\fiN\in \CkR[5]$.
(c) Errors of the localised needlet approximation (b).
(d) Errors of the hyperinterpolation approximation $\hyper$ with degree $L=128$.}\label{fig:dN.SHA_GRFhump_J7}
\end{figure}

\hspace{4mm}
\appendix

\section*{Acknowledgements}
The authors thank Yoshihito Kazashi for helping to improve the probabilistic framework of the paper. This research includes extensive computations using the Linux computational cluster Katana supported by the Faculty of Science, UNSW Australia.

%\section*{References}
\bibliography{srfdisN}

\end{document}